\documentclass[11pt]{amsart}
\usepackage{amscd, amssymb}
\usepackage{tikz}
\usetikzlibrary{arrows,positioning} 

\usepackage{comment}

\usepackage{amsthm}
\usepackage{mathrsfs}
\usepackage{bbm}
\usepackage{ stmaryrd }

\usepackage{multirow}
\usepackage{makecell}

\usepackage{hyperref}
\hypersetup{
	colorlinks=true,
	linkcolor=blue,
	filecolor=magenta,  
	urlcolor=cyan,
}

\def\End{\operatorname{End}}

\def\id{\operatorname{id}}

\def\Ind{\operatorname{Ind}}

\def\Hom{\operatorname{Hom}}

\def\C{\mathbb{C}}

\def\N{\mathbb{N}}

\def\G{\mathbb{G}}

\def\A{\mathbb{A}}

\def\A{\mathbb{A}}

\def\LL{\mathcal{L}}
\def\OO{\mathcal{O}}
\def\KK{\mathcal{K}}
\def\BB{\mathcal{B}}
\def\HH{\mathcal{H}}
\def\UU{\mathcal{U}}
\def\VV{\mathcal{V}}
\def\JJ{\mathcal{J}}
\def\MM{\mathcal{M}}

\def\GG{\mathcal{G}}
\def\QQ{\mathcal{Q}}
\def\FF{\mathcal{F}}
\def\DD{\mathcal{D}}
\def\NN{\mathcal{N}}
\def\PP{\mathcal{P}}

\def\XX{\mathcal{X}}
\def\YY{\mathcal{Y}}
\def\ZZ{\mathcal{Z}}

\def\Cs{\mathscr{C}}

\def\m{\mathfrak{m}}

\def\p{\mathfrak{p}}

\def\g{\mathfrak{g}}
\def\h{\mathfrak{h}}

\def\ol{\overline}

\def\Rep{\operatorname{Rep}}

\def\Spec{\operatorname{Spec}}

\def\Vec{\operatorname{Vec}}
\def\gr{\operatorname{gr}}

\def\ul{\underline}

\def\sub{\subseteq}

\newcommand{\cC}{\mathscr{C}}
\def\xto{\xrightarrow}

\def\onto{\twoheadrightarrow}

\def\gp{/_{\hspace{-0.2em}gp}}

\newtheorem{thm}{Theorem}[section]
\newtheorem{cor}[thm]{Corollary}
\newtheorem{lemma}[thm]{Lemma}
\newtheorem{prop}[thm]{Proposition}

\theoremstyle{definition}
\newtheorem{definition}[thm]{Definition}
\theoremstyle{remark}
\newtheorem{remark}[thm]{Remark}
\newtheorem{example}[thm]{Example}

\numberwithin{equation}{section}

\usepackage{relsize}
\usepackage{fancyhdr}
\usepackage[all,cmtip]{xy}

\begin{document}

	\title{Homogeneous spaces in tensor categories}
		\author{Kevin Coulembier, Alexander Sherman}


	\begin{abstract}
		Let $\mathscr{C}$ be a symmetric tensor category of moderate growth, and let $\HH\leq\GG$ be algebraic groups in $\mathscr{C}$.  We prove that the homogeneous space $\GG/\HH$ exists as a scheme and is of finite type when $\mathscr{C}$ is geometrically reductive and maximally nilpotent, conditions that are conjecturally equivalent to incompressibility.  A key tool is the introduction of a Frobenius kernel of a group scheme.  We further show that while $\GG_0/\HH_0$ and $(\GG/\HH)_0$ need not be the same, they are close enough, so that $\GG/\HH$ is quasi-affine/affine/proper if and only if $\GG_0/\HH_0$ is.
	\end{abstract}

	\maketitle
	\pagestyle{plain}

	\section{Introduction}
	Let $\Bbbk$ be an algebraically closed field.
	A (symmetric) tensor category in the sense of \cite{Del1, EGNO} is a $\Bbbk$-linear abelian symmetric rigid monoidal category satisfying some finiteness properties, and this concept axiomatises the representation categories of group schemes. A finiteness condition that is typically not included in the definition of a tensor category is that of moderate growth. This condition demands that the length of the tensor powers of any given object grows at most exponentially fast, and is clearly satisfied for representation categories of affine group schemes.

	A classical theorem of Deligne in \cite{Del2} states that, if the characteristic of $\Bbbk$ is zero, every tensor category over $\Bbbk$ of moderate growth is the representation category of an affine group superscheme (a supergroup). In other words, the study of tensor categories of moderate growth then becomes the study of representation theory of supergroups. An essential tool in the representation theory of affine group schemes is the quotient scheme $\GG/\HH$ of an affine group scheme $\GG$ by a subgroup $\HH\leq\GG$. A standard example is the flag variety $\GG/\BB$ for a reductive group $\GG$ with Borel subgroup $\BB$. By Deligne's theorem, it becomes imperative to prove the existence of such quotient superschemes for supergroups. This was achieved (in all characteristics) in \cite{MT, MZ}.  Homogeneous supervarieties were used crucially in \cite{Brundan,GS,PenkovSer} for computing characters of irreducible representations of supergroups, and more recently in \cite{SS,SSV,SV} to apply ideas from modular representation theory to representations of supergroups over $\C$.
	
	If the characteristic of $\Bbbk$ is $p>0$, then only partial results about the structure theory of tensor categories of moderate growth are available. It is a general fact that every tensor category of moderate growth admits a tensor functor to an `incompressible' tensor category, see \cite{CEO1}, where the latter categories are defined by the property that every tensor functor out of them is an embedding of a tensor subcategory. Deligne's theorem can then be rephrased as stating that the only incompressible tensor categories in characteristic zero are the categories of (super) vector spaces. In general, it thus follows that the study of tensor categories of moderate growth becomes equivalent to the representation theory of affine group schemes internal to incompressible tensor categories.

	In \cite{BEO}, it was conjectured that in characteristic $p>0$ the only incompressible tensor categories are the tensor subcategories of an incompressible category $\mathrm{Ver}_{p^\infty}=\cup_n \mathrm{Ver}_{p^n}$. A partial result in this direction is that every `Frobenius exact' tensor category of moderate growth admits a tensor functor to $\mathrm{Ver}_p$, see \cite{CEO2} and references therein. For the same reasons as in characteristic zero, it thus becomes important to study the existence of quotients $\GG/\HH$ in some scheme theory for $\mathrm{Ver}_{p^\infty}$, or (potentially) more general incompressible tensor categories. In \cite{C2, C3}, the basics of algebraic geometry in tensor categories satisfying some technical conditions (maximal nilpotency and geometric reductivity) were developed. These conditions seem a plausible intrinsic characterisation of incompressible categories; for instance, they are satisfied in $\mathrm{Ver}_{2^\infty}$. The question of whether homogeneous spaces $\GG/\HH$ exist as schemes in such tensor categories is thus well-defined, conjecturally equivalent to the corresponding question for $\mathrm{Ver}_{p^\infty}$, and we answer it affirmatively in the current paper. In particular, this proves existence of homogeneous spaces in $\mathrm{Ver}_{2^\infty}$ and $\mathrm{Ver}_p$.

	Henceforth, we work over an algebraically closed field $\Bbbk$ of characteristic $p>0$.  As explained above, the relevant geometric questions in characteristic 0 are already well understood. Our approach relies heavily on the idea of Frobenius twisting, only available in positive characteristic. For the case of super vector spaces, this specific positive characteristic method to study homogeneous spaces was outlined in \cite{MZ}. In our generality of tensor categories, we define Frobenius twists and kernels of algebraic groups. One of the crucial observations is that the quotient by a large enough Frobenius kernel is an ordinary algebraic group over $\Bbbk$.  This fundamental structural result of Proposition~\ref{prop twist alg gp}, that all finite type affine group schemes in $\cC$ are glued together from infinitesimal group schemes and ordinary group schemes over $\Bbbk$, should have applications beyond the current paper. Using the traditional formalism of quotients in the category of faisceaux \cite{DG} and the language of algebraic geometry in tensor categories \cite{C3} allows us to reduce the existence of quotient schemes to the classical case of algebraic groups over $\Bbbk$ and the infinitesimal case in $\cC$. We also derive applications of our results to equivariant sheaves and present a more explicit description of quotients for the important special case $\cC=\mathrm{Ver}_p$, following the approach of \cite{MT}.
	
	The geometric properties of the scheme $\GG/\HH$ are known to have important consequences in representation theory, especially for the induction functor $\operatorname{Ind}_{\HH}^{\GG}(-)$.  For this reason, it is important to know when $\GG/\HH$ is quasi-affine, affine, or proper, see for instance~\cite{C3}.  We show that $\GG/\HH$ is (quasi-)affine or proper if and only if $\GG_0/\HH_0$ is, where $\GG_0$ is the `body' of $\GG$, the underlying algebraic group which is generally more easily understood than $\GG$.  This result is known in the super setting, where in fact we have that $(\GG/\HH)_0=\GG_0/\HH_0$.  We find that the latter equality fails when $\mathscr{C}$ is not a tensor subcategory of $\operatorname{Ver}_p$,  see Example~\ref{example ver4+ homog space}.   Nevertheless, there is a natural closed embedding $\GG_0/\HH_0\to(\GG/\HH)_0$, and this is a universal homeomorphism, see Theorem \ref{Main} (2).

	The paper is organised as follows. In Section~\ref{Prel}, we recall and establish some basic properties of maximally nilpotent and geometrically reductive tensor categories, and scheme theory therein. In Section~\ref{AGS} we focus on affine group schemes in such categories. In Section~\ref{FrobTw} we introduce and study a notion of the Frobenius twist of schemes and the corresponding Frobenius kernels of affine group schemes. In Section~\ref{Faisc} we write out the classical idea of constructing homogeneous spaces via the theory of faisceaux in the generality we need, and we illustrate the technique in Section~\ref{Additive} by giving a direct construction of homogeneous spaces for additive group schemes. Section~\ref{MainSec} contains the main result about the existence of quotients $\GG/\HH$ and their essential properties. In Section~\ref{Affine} we discuss affine homogeneous spaces. Finally, in Section~\ref{Verp}, we give a more detailed description of homogeneous spaces for the important special case $\mathrm{Ver}_p$.

	\subsection{Acknowledgements} The authors would like to thank Alex Youcis for helpful discussions. We also would like to thank Akira Masuoka for pointing out an error in an earlier version of this text.  K.C. was supported by ARC grant FT220100125 and A.S. was supported by ARC grant DP210100251 and by an AMS-Simons Travel Grant.

	\section{Preliminaries}\label{Prel}
	We work throughout over an algebraically closed field $\Bbbk$ of characteristic $p>0$.  
	
	\subsection{Tensor categories}\label{section tens cats} 
	An essentially small, $\Bbbk$-linear symmetric monoidal category $(\mathscr{A},\otimes,\mathbf{1})$ will be called a tensor category (over $\Bbbk$) if
	\begin{enumerate}
		\item $\mathscr{A}$ is abelian,
		\item $\Bbbk\xto{\sim}\End(\mathbf{1})$,
		\item $(\mathscr{A},\otimes,\mathbf{1})$ is rigid, and
		\item each object of $\mathscr{A}$ is of finite length.
	\end{enumerate}
	Given an object $X$ in a tensor category $\mathscr{A}$, we write $X^*$ for its dual and $S(X)$ for its symmetric algebra in the ind-completion of $\mathscr{A}$.
	
	Throughout this paper, we will work over a fixed tensor category $\mathscr{C}_{fin}$, and we write \linebreak $\mathscr{C}:=\Ind\mathscr{C}_{fin}$ for the ind-completion of $\mathscr{C}_{fin}$. By an abuse of terminology, we will also refer to $\cC$ as a tensor category.  The full subcategory $\mathscr{C}_{fin}\sub\mathscr{C}$ consists exactly of the objects of $\mathscr{C}$ of finite length, or equivalently those which are compact, or yet equivalently rigid.  Note that $\mathscr{C}$ is a symmetric monoidal category with an exact tensor product, is abelian, and every object is a union of its finite length subobjects.  The classical case of our considerations is $\cC=\mathrm{Vec}$, the category of all vector spaces over $\Bbbk$.

    We will use the term tensor functor to refer to exact $\Bbbk$-linear symmetric monoidal functors between tensor categories, or the cocontinuous extensions of those functors to the ind-completions.
    We have the tensor functor $\Vec\to \cC$, $V\mapsto V\otimes_{\Bbbk}\mathbf{1}$, and we will freely use its lax monoidal left adjoint $\cC\to\Vec$ which takes the maximal `invariant' subobject (or subalgebra in case we apply it to an algebra). Concretely,
	for $X\in\mathscr{C}$, we will write $X_{(0)}=\Hom(\mathbf{1},X)$, viewed canonically as a subobject of $X$.  
    
    For $X,Y\in\mathscr{C}$, we write $\ul{\Hom}(X,Y)$ for the internal $\Hom$ in $\mathscr{C}$, which is defined by the adjunction
	\[
	\Hom(Z\otimes X,Y)\cong\Hom(Z,\ul{\Hom}(X,Y)).
	\]
	Note that $\ul{\Hom}$ gives a bifunctor $\mathscr{C}^{op}\times\mathscr{C}\to\mathscr{C}$. We will also write $X^\ast:=\ul{\Hom}(X,\mathbf{1})$, which coincides with the dual whenever $X\in\cC_{fin}$.

     Let $R$ be an algebra in $\mathscr{C}$ (by this we technically mean a monoid object in $\Cs$) and let $S\sub R$ be a subalgebra.  Then for a right $S$-module $N$ and a right $R$-module $M$, we define the right $R$-module $\ul{\Hom}_S(R,N)$ by
	\begin{equation}\label{eqn tensor hom ring}
		\Hom_{R}(M,\ul{\Hom}_{S}(R,N))\cong \Hom_S(M,N).
	\end{equation}
	Alternatively, one may define $\ul{\Hom}_{S}(R,N)$ by the equalizer:
	\[
	\ul{\Hom}_{S}(R,N)\to \ul{\Hom}(R,N)\rightrightarrows\ul{\Hom}(R\otimes S,N).
	\]
	In this way we see that $\ul{\Hom}_{S}(R,N)$ may be viewed as a subobject of $\ul{\Hom}(R,N)$.  
    \subsubsection{Properties GR and MN}     
	{\em Henceforth we will fix a tensor category $\mathscr{C}$ and, as explained and motivated in the introduction, assume that $\mathscr{C}$ satisfies the following restrictive assumptions introduced in \cite{C2}:}
	\begin{enumerate}
		\item[(GR)] a non-zero morphism $\alpha:X\to \mathbf{1}$ is split on some symmetric power of $\alpha$.
		\item[(MN1)] if $L$ is a simple nontrivial module, then $S(L)$ is finite.
		\item[(MN2)] if $\alpha:\mathbf{1}\to X$ is nonsplit, then some symmetric power of $\alpha$ is 0.
	\end{enumerate}
	Note that in \cite{C2} it was assumed that the object $X$ in (GR) and (MN2) is compact, but it is straightforward to show that this is equivalent to having the property hold for all $X\in\mathscr{C}$.

	\begin{example}
		If $\GG$ is a classical affine algebraic group, then $\operatorname{Rep}\GG$ is a tensor category which satisfies (GR) if $\GG$ is (geometrically) reductive. In contrast, the combination (MN1-2) is only satisfied when $\GG$ is trivial.
	\end{example}

	\begin{example}
		Assume $\Bbbk$ is not of characteristic $2$, and let $\mathscr{C}=\operatorname{SVec}$, the category of super vector spaces.  Then $\mathscr{C}$ is semisimple and admits two simple objects, $\mathbf{1}$ and $\bar{\mathbf{1}}$.  We have $S(\bar{\mathbf{1}})=\mathbf{1}\oplus\bar{\mathbf{1}}$.  From these observations it is easy to check that $\mathscr{C}$ satisfies (GR) and (MN1-2). 
	\end{example}
	
	\begin{example}\label{example ver_p}
		Let $\mathscr{C}=\operatorname{Ver}_p$ (see \cite{O}),   
		the Verlinde-$p$ category.  Explicitly, $\operatorname{Ver}_p$ may be defined as the ind-completion of the semisimplification of the category of finite-dimensional $C_p$-modules, where $C_p$ is the cyclic group of order~$p$.  Recall that $\operatorname{Ver}_p$ is semisimple, has $p-1$ simple objects $L_1=\mathbf{1},\dots,L_{p-1}$,  and satisfies $S^{k+1}(L_{p-k})=0$ for $0<k<p-1$ (\cite[Prop.~2.4]{EOV}), so that (GR) and (MN1-2) hold.  Notice that $\operatorname{SVec}$ is the full subcategory of $\operatorname{Ver}_p$ generated by $L_{1}$ and $L_{p-1}$ when $p>2$.  Further, if we write $\operatorname{Ver}_p^+$ for the full subcategory of $\operatorname{Ver}_p$ generated by $L_i$ for $i$ odd, then $\operatorname{Ver}^+_p$ is a tensor subcategory of $\operatorname{Ver}_p$, and thus also satisfies (GR) and (MN1-2).

        For an object $X\in\operatorname{Ver}_p$, we write $X=X_{(0)}\oplus X_{\not=0}$, where $X_{\not=0}$ is a direct sum of $L_i$, $i\not=1$.
	\end{example}
	
	We refer to \cite{BE, BEO} for details on the incompressible categories $\operatorname{Ver}_{p^n}$ with $n\in\mathbb{Z}_{>0}\cup\{\infty\}$.
	
	\begin{example}\label{example ver4+}
		The category $\mathscr{C}:=\operatorname{Ver}_4^+$ is described, as a monoidal category, as modules over the Hopf algebra $\Bbbk[x]/x^2$.  The braiding $V\otimes W\to W\otimes V$ is given by $v\otimes w\mapsto w\otimes v+xw\otimes xv$. Since $\operatorname{Ver}_4^+$ has no nontrivial simple objects, (MN-1) holds trivially. 
		To check (GR) and (MN-2), we see that $\mathscr{C}$ has a unique nontrivial indecomposable $P$, and there is a short exact sequence 
		\[
		0\to \mathbf{1}\to P\to\mathbf{1}\to 0.
		\]
		Writing $P=\Bbbk\langle u,v\rangle$, where $x\cdot v=u$, one may compute that:
		\[
		S(P)=\Bbbk[u,v]/u^2
		\]
		Thus the image of the second symmetric power of $\mathbf{1}\to P$ vanishes, and similarly the second symmetric power of $P\to\mathbf{1}$ splits.  From this we may deduce that (GR) and (MN-2) hold. More generally, it was proven in \cite[\S 9]{CEO1} that $\operatorname{Ver}_{2^\infty}$ satisfies (GR) and (MN1-2).
	\end{example}
	
	\begin{remark}
		In \cite{C2}, it was conjectured that $\operatorname{Ver}_{p^\infty}$ satisfies (GR) and (MN1-2) also when $p>2$.  Indeed, the expectation (or hope) is that all incompressible, moderate growth tensor categories satisfy (GR) and (MN1-2). We also refer to \cite[\S 3.2]{C2} for partial results on showing that (MN1-2) implies incompressiblity.
	\end{remark}

	\begin{remark}  The assumptions (GR) and (MN1-2) imply the geometry (developed in \cite{C3}) should share characteristics similar to supergeometry, where there is an underlying classical geometry along with some extra infinitesimal `odd' parts that come along for the ride.  In particular, this setting leaves the door open for a  general theory of Harish-Chandra pairs for algebraic groups, which was developed in the $\operatorname{Ver}_p$ case in \cite{V1} and in the super case in several papers (e.g.~\cite{MS}).  Preliminary work suggests that a theory of Harish-Chandra pairs may hold in these categories, but it will be weaker in the sense that one has to replace the Lie algebra with an infinitesimal group larger than the one determined by the small enveloping algebra.  Developing such a theory will be taken up in future work.
	\end{remark}

	\begin{definition}
		For an object $X$ in $\mathscr{C}$, let us write $X_{nil}$ for the joint kernel of all maps to $\mathbf{1}$.  This defines an endofunctor of $\mathscr{C}$, which is neither monoidal nor exact; however we always have
		\begin{equation}\label{eqn nil}
		    (X\otimes Y)_{nil}\sub X_{nil}\otimes Y+X\otimes Y_{nil}.
		\end{equation}
	\end{definition}

	\subsection{Commutative algebra}  Fundamentals of commutative algebra and scheme theory in $\mathscr{C}$ were developed in \cite{C2} and \cite{C3}.  Throughout this paper, if $A$ is a commutative algebra in $\mathscr{C}$ then we will write $m_{A}=m$ for its multiplication morphism and $\eta_{A}=\eta:\mathbf{1}\to A$ for its unit.  We say that a commutative algebra $A$ is finitely generated if there exists a compact object $X$ in $\mathscr{C}$ and an algebra morphism $S(X)\onto A$ that is an epimorphism in $\cC$. A consequence of our assumptions on $\mathscr{C}$ is that finitely generated algebras are Noetherian, see \cite[Lem.~6.4.4]{C2}, so in particular ideals are finitely generated.

	From now on, an `algebra' will always refer to an algebra in $\mathscr{C}$. When we want specifically to refer to an ordinary algebra we will say $\Bbbk$-algebra or algebra in $\mathrm{Vec}$. The same rule will be followed for (affine group) schemes.

    Every commutative algebra $A$ in $\Cs$ has a natural quotient algebra and subalgebra in $\mathrm{Vec}$, as follows.  Write $J=J_A$ for the ideal generated by $A_{nil}\sub A$, and set $\bar{A}:=A/J$.  Then $\bar{A}$ and $A_{(0)}$ will be commutative algebras in $\mathrm{Vec}$, and we have morphisms of algebras
    \begin{equation}\label{eqn natl alg morphisms}
            A_{(0)}\hookrightarrow A\onto \bar{A}.
    \end{equation}
   Moreover, the assignments $A\mapsto A_{(0)}$ and $A\mapsto\bar{A}$ are functorial, and the morphisms in (\ref{eqn natl alg morphisms}) define natural transformations.
	
    \subsubsection{Artin-Rees lemma} 
    \begin{lemma}\label{lemma artin rees krull}
Let $I$ be an ideal of a commutative Noetherian algebra $R$ in $\cC$ and $M$ a finitely generated $R$-module.
\begin{enumerate}  
    \item  The Artin-Rees lemma holds in $\Cs$:  if $N\sub M$ is a submodule, then there exists $k\geq0$ such that for $n\geq k$
    \[
    I^nM\cap N=I^{n-k}(I^kM\cap N).
    \]
    \item The Krull intersection theorem holds in $\Cs$:  if $R$ is a local algebra and $I$ is a proper ideal of $R$, then $\bigcap\limits_{n=1}^\infty I^n=0$.
\end{enumerate}
\end{lemma}
\begin{proof}
    (1) Indeed, the proof of the classical case in \cite[Corollary~10.10]{AM} using the Rees algebra $\oplus_{n\ge 0}I^n$ carries over verbatim using (MN1) by \cite[Lem.~6.4.4]{C1} which shows that the Rees algebra (being finitely generated) is Noetherian. 

    (2) Applying (1) to the case of $N=\bigcap\limits_{n=1}^\infty I^n\sub M=R$, we obtain that $IN=N$.  By Nakayama's lemma for $\Cs$ (see \cite[Cor.~6.1.4]{C1}), we obtain $N=0$.
\end{proof}

	\subsubsection{Finite morphisms} We say that a morphism $A\to B$ of commutative algebras is finite if $B$ is a finitely generated $A$-module. 

    \begin{prop}\label{Cor:AM}
    Let $A$ be a finitely generated commutative algebra.
		\begin{enumerate}
			\item $A$ is finite over $A_{(0)}$;
			\item $A_{(0)}$ is a finitely generated algebra.
		\end{enumerate}    
    \end{prop}	

    For the proof of Proposition \ref{Cor:AM}, we will need several lemmas which will be useful later on.
    
	\begin{lemma}\label{lemma finite subalg}
		Let $A$ be a finitely generated commutative algebra with nilpotent ideal $I$, and write $C:=A/I$.  If $\psi:B\to A$ is an algebra morphism such that the composition $B\to C$ is finite, then $\psi$ is finite.
	\end{lemma}
	
	\begin{proof}
		Let $X\sub C$ be a compact object such that $B\otimes X\to C$ is an { epimorphism}.  Choose a compact object $X'\sub A$ such that $X\sub X'/I\cap X'$,  so that we have an  epimorphism $B\otimes X'\to C$.  Because $A$ is finitely generated, so is $I$, and thus we may choose a compact object $Y\sub I$ such that $C\otimes Y\to I/I^2$ is an { epimorphism}.   In particular, the composition 
        \[
        B\otimes X'\otimes S^n(Y)\to C\otimes S^n(Y)\to I^n/I^{n-1}
        \]
        is an  epimorphism for all $n$.  
        
        We obtain a natural map $\phi:B\otimes X'\otimes S^\bullet(Y)\to A$, which we claim is an  epimorphism.  Indeed, $B\otimes X'\otimes S^\bullet(Y)$ has a filtration coming from the grading on $S^\bullet(Y)$, and $A$ has a filtration defined by the powers of $I$. The map $\phi$ preserves this filtration, and the induced map on the associated graded is $B\otimes X'\otimes S^\bullet(Y)\to \bigoplus\limits_{n\in\N}I^n/I^{n-1}$, which on the $n$th graded component is the map $B\otimes X'\otimes S^n(Y)\to I^n/I^{n-1}$, which we have shown is an  epimorphism.    Since the associated graded of $\phi$ is an  epimorphism, $\phi$ is as well.   Because $I$ is nilpotent, $S^\bullet(Y)$ is compact, and thus so is $X'\otimes S^\bullet(Y)$.  This shows $A$ is finite over $B$.
	\end{proof}

	\begin{cor}
		Let $A,B$ be finitely generated commutative algebras.  Then a morphism $f:A\to B$ is finite if and only if the induced morphism $\bar{f}:\bar{A}\to\bar{B}$ is finite.
	\end{cor}
	\begin{proof}
		The forward direction is clear.  For the backward direction, we obtain that $A\to\bar{B}$ is a finite morphism.  Thus we may apply Lemma \ref{lemma finite subalg}.
	\end{proof}
	
	The following is essentially \cite[Prop.~7.8]{AM}.
	\begin{lemma}\label{lemma atiyah-mac}
		Suppose that $A$ is a finitely generated commutative algebra and $B\sub A$ is a subalgebra such that $A$ is finite over $B$.  Then $B$ is also finitely generated.
	\end{lemma}
	\begin{proof}
		Let $X\sub A$ be a compact object which generates $A$ as an algebra, and let $Y\sub A$ be a compact object such that $A=BY$.  Then choose a compact object $Z\sub B$ such that $ZY$ contains $X$ and $Y^2$.  Finally, write $B'$ for the subalgebra of $B$ generated by $Z$.  
		
		We claim that $B'Y$ equals $A$, so that $A$ is finite over $B'$.  Indeed, since $Y^2\sub B'Y$ and $Z\sub B'$, we see that:
		\[
		X^n\sub (ZY)^n\sub B'Y.
		\]
		Since $B'$ is Noetherian and $B$ is a $B'$-submodule of $A$, we may conclude $B$ is also finite over $B'$, and is therefore also finitely generated.
	\end{proof}
	
	\begin{proof}[Proof of Proposition~\ref{Cor:AM}]
		By Lemma~\ref{lemma atiyah-mac}, (2) follows from (1). For (1), by Lemma~\ref{lemma finite subalg} it suffices to show that $\overline{A}$ is finite over $A_{(0)}$. For this we can choose a finite set of generators $x_1,\dots,x_r\in \overline{A}$ and observe that by condition (GR), there exists an $N\in\N$ such that for all $i$, $x_i^N$ are in the image of $A_{(0)}\to \overline{A}$.
	\end{proof}
	
	\begin{remark}
		In \cite[Conjecture~3.1.6]{C2} it was conjectured that (2) of Proposition~\ref{Cor:AM} holds with solely the assumption that $\mathscr{C}$ satisfies (GR), so that this condition is equivalent to (GR), see \cite[Thm. 3.1.5]{C2}.
	\end{remark}

    \subsection{Scheme theory}
    
	Let $\operatorname{Lrs}_{\mathscr{C}}$ denote the category of locally $\mathscr{C}$-ringed spaces $\XX=(|\XX|,\OO_{\XX})$.  Here we write $|\XX|$ to refer to the underlying topological space of $\XX$, although we  will often employ an abuse of notation and simply write $\XX$ in place of $|\XX|$.   
	
	Schemes in $\mathscr{C}$ were introduced in \cite{C3}.  A scheme is a locally $\mathscr{C}$-ringed space $\XX$ that is locally isomorphic to an affine scheme.  Affine schemes are defined by a Spec construction similarly to the classical case.  If $\XX$ is a scheme and $\UU\sub\XX$ is an open subscheme, then we will write $\Bbbk[\UU]:=\OO_{\XX}(\UU)$, and this will be a commutative algebra in $\mathscr{C}$.  We say that a scheme $\XX$ is algebraic if it is of finite type, i.e.~there exists a finite affine covering $\XX=\bigcup\limits_{i}\UU_i$ such that $\Bbbk[\UU_i]$ is finitely generated for all $i$.   Note that the fibre product of algebraic schemes over an algebraic scheme is still algebraic.

	  \subsubsection{Schemes in $\Vec$ associated to $\XX$}  Under our assumptions, the map $\Spec\bar{A}\to\Spec A$ is a closed embedding (immersion) which induces the identity map on topological spaces (\cite[Thm.~5.3.1]{C2}).  On the other hand, $A_{(0)}$ is a subalgebra of $A$, and by \textit{loc.~cit.}, the inclusion $A_{(0)}\sub A$ induces a natural morphism $\operatorname{Spec}A\to\Spec A_{(0)}$ which is a homeomorphism of topological spaces.
	
	In general, for a scheme $\XX$, we write $\JJ_{\XX}$ for the ideal sheaf where $\JJ_{\XX}(\UU)=J_{\OO_{\XX}(\UU)}$.  Set $\XX_0$ to be the closed subscheme of $\XX$ determined by $\JJ_{\XX}$.  We may also view $\XX\mapsto\XX_0$ as a functor, given by the right adjoint to the inclusion $\operatorname{Lrs}_{\Vec}\to \operatorname{Lrs}_{\mathscr{C}}$.   
  
    We define $\XX_{(0)}$ to be the locally ringed space $(|\XX|,(\OO_{\XX})_{(0)})$.  This is a scheme, as affine locally it is isomorphic to $\Spec A_{(0)}$.  The assignment $\XX\mapsto\XX_{(0)}$ is functorial, and is the left adjoint to the inclusion $\operatorname{Lrs}_{\Vec}\to \operatorname{Lrs}_{\mathscr{C}}$.  

    Summarising, we have the following natural morphisms of schemes, which are the identity on underlying topological spaces:
    
    \begin{equation}\label{eqn natl morphism of schemes}
\XX_{0}\hookrightarrow\XX\to\XX_{(0)} .
    \end{equation}

    The morphisms in (\ref{eqn natl morphism of schemes}) correspond to the maps of algebras in (\ref{eqn natl alg morphisms}) for any affine open subscheme $\Spec A\sub \XX$.  
    We say that $\XX$ is \emph{purely even} if $\XX=\XX_0$, or equivalently the maps in (\ref{eqn natl morphism of schemes}) are identity maps.  
    
    Given a morphism of schemes $f:\XX\to\YY$, we obtain a commutative diagram:
     \begin{align}\label{diagram map of schemes}
       \xymatrix{
       \XX_0\ar[r] \ar[d]^{f_0} & \XX\ar[d]^f\ar[r] & \XX_{(0)}\ar[d]^{f_{(0)}} \\
       \YY_0\ar[r] & \YY\ar[r] & \YY_{(0)}.       }
\end{align}
The morphisms $f,f_{0}$, and $f_{(0)}$ induce the same map on underlying topological spaces.

	\subsection{Properties of schemes and morphisms}
	
	Recall that a morphism of schemes $f:\XX\to\YY$ is a universal homeomorphism if any base change of $f$ is a homeomorphism.  The following is proven in \cite[\href{https://stacks.math.columbia.edu/tag/0CNF}{Lem.~29.45.8}]{Stacks}.
    
	\begin{lemma}\label{lemma univ homeo prep}\label{lemma univ homeo}
		Suppose $\XX,\YY,\ZZ$ are algebraic schemes with morphisms $f:\XX\to\YY$, $g:\YY\to\ZZ$. If any two of $f,g,g\circ f$ are universal homeomorphisms, so is the third.
	\end{lemma}
	
	The following is \cite[\href{https://stacks.math.columbia.edu/tag/0CNF}{Lem.~29.46.9}]{Stacks}.
	\begin{lemma}\label{lemma univ homeo stacks}
		Suppose that $\mathscr{C}=\operatorname{Vec}$ and $\phi:A\to B$ is a morphism of commutative algebras.  The following are equivalent:
		\begin{enumerate}
			\item The induced morphism $\Spec B\to\Spec A$ is a universal homeomorphism.
			\item The kernel of $\phi$ is a locally nilpotent ideal and for every $b\in B$ there exists $n>0$ such that $b^{p^n}\in\operatorname{Im}\phi$.
		\end{enumerate}
	\end{lemma}
	
	\begin{lemma}\label{lemma even embedding univ homeo}
		Let $\XX,\YY$ be algebraic schemes.
		\begin{enumerate}
			\item The morphism $\XX_0\to\XX$ is a finite, universal homeomorphism.
			\item The morphism $\XX\to\XX_{(0)}$ is a finite, universal homeomorphism.
            \item A morphism $f:\XX\to\YY$ is a universal homeomorphism if and only if $f_0$ is, and if and only if $f_{(0)}$ is.
 
		\end{enumerate}
	\end{lemma}
	\begin{proof}
		We can prove (3) assuming (1) and (2) by applying Lemma \ref{lemma univ homeo prep}.
		
		To prove (1) and (2), the property of being a universal homeomorphism is Zariski local  on the base \cite[\href{https://stacks.math.columbia.edu/tag/0CEX}{Lem.~35.23.11}]{Stacks}, so we may assume $\XX$ is affine, given by the spectrum of a commutative algebra $A$.  The morphism $A\to\bar{A}$ is the quotient by a nilpotent ideal which is always a universal homeomorphism.
		
		For (2), we may now invoke the case of $f_0$ for (3) which tells us that it suffices to show $\Spec \bar{A}\to\Spec A_{(0)}$ is a universal homeomorphism.  However by (GR), the morphism $A_{(0)}\to\bar{A}$ satisfies (2) of Lemma \ref{lemma univ homeo stacks}.  Finally, $A_{(0)}\to A$ is finite by Proposition~\ref{Cor:AM}.
	\end{proof}

	\begin{lemma}\label{lemma univ homeo fibre prods}
		Let $\XX,\YY$, and $\ZZ$ be algebraic schemes with morphisms $\XX\to\ZZ$, $\YY\to\ZZ$.   Then $(\XX\times_{\ZZ}\YY)_0=\XX_0\times_{\ZZ_0}\YY_0$, and the following are finite, universal homeomorphisms:
        \begin{enumerate}
            \item $\XX_0\times_{\ZZ_0}\YY_0\to\XX\times_{\ZZ}\YY$.
            \item $(\XX\times_{\ZZ}\YY)_{(0)}\to\XX_{(0)}\times_{\ZZ_{(0)}}\YY_{(0)}$.
            \item $\XX\times_{\ZZ}\YY\to\XX_{(0)}\times_{\ZZ_{(0)}}\YY_{(0)}$.
        \end{enumerate}     In particular, the natural morphisms $\XX_0\times_{\ZZ_0}\YY_0\to \XX\times_\ZZ\YY\to\XX_{(0)}\times_{\ZZ_{(0)}}\YY_{(0)}$ are homeomorphisms.
	\end{lemma}
	
	\begin{proof}
         The equality $(\XX\times_{\ZZ}\YY)_0=\XX_0\times_{\ZZ_0}\YY_0$ holds because the functor $\XX\mapsto\XX_0$ preserves limits, and thus fibre products.  Since $\XX_0\to \XX$ is a closed embedding for any scheme $\XX$, morphism (1) is a finite, universal homeomorphism by (1) of Lemma \ref{lemma even embedding univ homeo}. 
        
        Let us check morphism (2) is finite. This will show that morphism (3) is also finite, being the composition of morphism (2) and the finite morphism $\XX\times_{\ZZ}\YY\to(\XX\times_{\ZZ}\YY)_{(0)}$ (Lemma \ref{lemma even embedding univ homeo}, (2)). We may verify this affine locally, where it reduces to the morphism $A_{(0)}\otimes_{C_{(0)}}B_{(0)}\to (A\otimes_{C}B)_{(0)}$.  Because $A,B$ are finite over $A_{(0)},B_{(0)}$ respectively, we obtain that $A\otimes_{C_{(0)}}B$ is finite over $A_{(0)}\otimes_{C_{(0)}}B_{(0)}$. Thus using the epimorphism $A\otimes_{C_{(0)}}B\to A\otimes_CB$, we obtain that $A\otimes_CB$ is finite over $A_{(0)}\otimes_{C_{(0)}}B_{(0)}$.  Since $A_{(0)}\otimes_{C_{(0)}}B_{(0)}$ is Noetherian, the subalgebra $(A\otimes_CB)_{(0)}$ of $A\otimes_CB$ must be finitely generated over $A_{(0)}\otimes_{C_{(0)}}B_{(0)}$, which proves finiteness.  
    
		To show (2) and (3) are universal homeomorphisms, it suffices to show (3) is by Lemma \ref{lemma univ homeo prep}.  We once again may check this affine locally, and by (2) of Lemma \ref{lemma even embedding univ homeo}, it suffices to consider the morphism $A_{(0)}\otimes_{B_{(0)}}C_{(0)}\to\bar{A}\otimes_{\bar{B}}\bar{C}$.  Once again Lemma \ref{lemma univ homeo stacks} implies our result using the (GR) property.
	\end{proof}
	
	 We can define separated, affine, quasi-affine, finite, and proper morphisms of schemes in $\mathscr{C}$ in the exact way as for schemes in $\Vec$.

	\begin{prop}\label{prop rel properties}
		Let $\XX,\YY$ be algebraic schemes, and let $f:\XX\to\YY$ be a morphism.  Let (P) denote one of the properties of being separated, affine, quasi-affine, finite, proper, or a universal homeomorphism.  The following are equivalent:
		\begin{enumerate}
			\item $f:\XX\to\YY$ satisfies (P);
			\item $f_0:\XX_0\to\YY_0$ satisfies (P);
			\item $f_{(0)}:\XX_{(0)}\to\YY_{(0)}$ satisfies (P).
		\end{enumerate}
	\end{prop}
    
	\begin{proof}
           
       We begin by recalling that a morphism is finite if and only if it is affine and proper (\cite[\href{https://stacks.math.columbia.edu/tag/01WN}{Lem.~29.44.11}]{Stacks}), thus we need not prove the statement for finiteness if we show it for properness and affinity.  The case when (P) is the property of being a universal homeomorphism is proven in Lemma \ref{lemma even embedding univ homeo}.
       
       If (P) is the property of being (quasi-)affine, we may apply the commutative diagram (\ref{diagram map of schemes}) to reduce the question to showing that $\XX$ is (quasi-)affine if and only if $\XX_0$ is, and if and only if $\XX_{(0)}$ is.  In \cite[Remark~5.4.3]{C3} it was shown that $\XX$ is quasi-affine if and only if $\XX_{(0)}$ is. Using the GR property, we see that $\XX$ is quasi-affine if and only if $\XX_0$ is by \cite[Prop.~5.4.2(5)]{C3}. It is clear that if $\XX$ is affine then $\XX_{(0)}$ is as well. If $\XX_{(0)}$ is affine, then we already know that $\XX$ is quasi-affine, but also that the open immersion in \cite[Proposition~5.4.2(2)]{C3} must be surjective, showing that $\XX$ is affine.  If $\XX$ is affine then $\XX_0$ is as well, and the converse is a special case of \cite[Lemma~5.3.4]{C3}.

         If (P) is the property of being separated, we may use Lemma \ref{lemma univ homeo fibre prods} which tells us that the topological spaces of the three fibre products $\XX_0\times_{\YY_0}\XX_0$, $\XX\times_\YY\XX$, and $\XX_{(0)}\times_{\YY_{(0)}}\XX_{(0)}$ are naturally identified.  This implies the image of the diagonal embedding as a topological space will be the same in each case, which proves the claim.  
		
		Finally let (P) be the property of being proper.  If $\ZZ\to\YY$ is a morphism of schemes, we may again apply Lemma \ref{lemma univ homeo fibre prods} to obtain that $\XX\times_{\YY}\ZZ\to \ZZ$ is closed if and only if $\XX_0\times_{\YY_0}\ZZ_0\to\ZZ_0$ is, and if and only if $\XX_{(0)}\times_{\YY_{(0)}}\ZZ_{(0)}\to\ZZ_{(0)}$ is.  Thus the claim follows. 
	\end{proof}
	   We say a scheme $\XX$ is separated, respectively proper, if the corresponding structure map $\XX\to\Spec\Bbbk$ is.  
    \begin{cor}\label{prop scheme properties}  Let $\XX$ be an algebraic scheme, and let (P) stand for one of the properties of being separated, affine, quasi-affine, or proper.  Then the following are equivalent.
		\begin{enumerate}
			\item $\XX$ satisfies (P);
			\item $\XX_0$ satisfies (P);
			\item $\XX_{(0)}$ satisfies (P).
		\end{enumerate}
	\end{cor}

	\begin{cor}\label{cor proper pushforward}
		If $f:\XX\to\YY$ is a proper morphism of algebraic schemes, then the pushforward of a coherent sheaf remains coherent.
	\end{cor}
	
	\begin{proof}
		By Proposition \ref{prop rel properties}, $f_{(0)}:\XX_{(0)}\to\YY_{(0)}$ is proper.  Now we may use that $\YY\to\YY_{(0)}$ is finite, and thus a quasi-coherent sheaf on $\YY$ is coherent if and only if its pushforward to $\YY_{(0)}$ is coherent.
	\end{proof}

	\subsection{Infinitesimal schemes}\label{section infinitesimal schemes}
	
	\begin{definition}\label{defn infinitsml}
		We say an algebraic scheme $\XX$ is infinitesimal if $|\XX|$ is a singleton.
	\end{definition}
	
	It is clear that an infinitesimal scheme $\XX$ must be affine, since $\XX$ has an affine cover.
	We say that an algebra is local if it admits a unique maximal ideal.

	\begin{lemma}\label{lemma inftsml is local fin length} The following are equivalent for a finitely generated commutative algebra $A$.
		\begin{enumerate}
			\item $\Spec A$ is infinitesimal.
			\item $\Spec A_{(0)}$ is infinitesimal.
			\item $\Spec\bar{A}$ is infinitesimal.
			\item $A$ is a local algebra.
			\item $A$ is a finite-length (as an object in $\cC$), local algebra.
			\item $A$ has a unique prime ideal, and it is nilpotent.
		\end{enumerate}
	\end{lemma}
	\begin{proof}
		(1)$\iff$(2)$\iff$(3) By \cite{C2} we have identifications $|\Spec\bar{A}|=|\Spec A_{(0)}|=|\Spec A|$.  Note that by (2) of Proposition \ref{Cor:AM}, $A_{(0)}$ is finitely generated.
		
		(1)$\Rightarrow$(4) This is obvious.
		
		(4)$\Rightarrow$ (5) Observe that $\bar{A}$ will again be a local algebra finite-dimensional algebra by assumption, and is therefore finite dimensional.  Hence we obtain that the composition of maps $\mathbf{1}\to A\to \bar{A}$ is finite, allowing us to conclude by Lemma \ref{lemma finite subalg}.

		(5)$\Rightarrow$(6) Write $\m$ for the unique maximal ideal of $A$.  Then since $A$ is finite length, for some $n>0$ we must have $\m^n=\m^{n+1}$.  By Thm.~6.1.5 of \cite{C2}, Nakayama's lemma holds, and thus we have $\m^n=0$.
In particular $\m$ is included in any prime ideal and thus the only prime ideal.
        
		(6)$\Rightarrow$(1) This is clear.
	\end{proof}
    \begin{example}\label{ex:finloc}
        Consider a commutative algebra $A\in\cC_{fin}$. Then it is a finite product of local algebras. Indeed, we can apply this classical result to the $\Bbbk$-algebra $A_{(0)}$, use the corresponding orthonormal idempotents of $A_{(0)}$ to find a product $A=\prod_i A_i$ with each $(A_i)_{0}$ local, so that $A_i$ is local by Lemma~\ref{lemma inftsml is local fin length}.
    \end{example}

	\section{Affine group schemes} \label{AGS}
    Recall that throughout we assume that $\cC$ is a (symmetric) tensor category satisfying (GR) and (MN1-2).

    \subsection{Algebraic groups}
	An affine group scheme is a representable functor from the category of commutative algebras in $\mathscr{C}$ to the category of groups.  An affine group scheme $\GG$ is represented by a commutative Hopf algebra $\Bbbk[\GG]$ in $\mathscr{C}$.  An \textbf{algebraic group} in $\mathscr{C}$ is an affine group scheme $\GG$ such that $\Bbbk[\GG]$ is finitely generated. We write $\Delta=\Delta_{\GG}:\Bbbk[\GG]\to\Bbbk[\GG]\otimes\Bbbk[\GG]$ for the comultiplication morphism on $\Bbbk[\GG]$, $\varepsilon=\varepsilon_{\GG}:\Bbbk[\GG]\to\mathbf{1}$ for the counit, and $\eta=\eta_{\GG}:\mathbf{1}\to\Bbbk[\GG]$ for the unit.  Set $\m_{\varepsilon}=\Bbbk^+[\GG]$ to be the augmentation ideal, i.e.~the kernel of $\varepsilon$.    Unless further specified, for the rest of the section $\GG$ is some fixed algebraic group.

    A homomorphism $\GG\to\HH$ of algebraic groups is a morphism of schemes such that $\GG(A)\to\HH(A)$ is a homomorphism of groups for all commutative algebras $A$.  Equivalently, the pullback morphism $\Bbbk[\HH]\to\Bbbk[\GG]$ is a morphism of Hopf algebras.
	
	\subsubsection{$\GG$-modules }A left $\GG$-module $M$ will be synonymous with a right $\Bbbk[\GG]$-comodule, and we use the symbol $\rho_{M}=\rho:M\to M\otimes\Bbbk[\GG]$ for the coaction. 
	Given a $\GG$-module $M$, write $M^{\GG}$ for the subobject of $\GG$-invariants, which is defined by the following equalizer for $\rho_M$ and $\id_M\otimes\eta$:
	\[
	M^{\GG}\to(M\rightrightarrows M\otimes\Bbbk[\GG]).
	\]
	We write $\operatorname{Rep}\GG$ for the category of right $\Bbbk[\GG]$-comodules (not necessarily of finite-length) in $\cC$.  Note that $\operatorname{Rep}\GG$ is once again the ind-completion of a tensor category, see \cite[\S 7.3]{C2}.

	\subsection{Subgroups, factor groups, and images}\label{sec subgrps} We refer to \cite[Sec.~7]{C2} for details on this section.
	
	A subgroup $\HH\leq\GG$ is a  representable subfunctor of $\GG$, and it corresponds to a quotient Hopf algebra of $\Bbbk[\GG]$.  By \cite[Lem.~7.1.3]{C2}, subgroups of $\GG$ are in bijective correspondence with monomorphisms $\HH\to\GG$ of algebraic groups up to isomorphism.  We say that $\HH$ is normal if $\HH(A)$ is normal in $\GG(A)$ for all commutative algebras $A$.  If $\HH$ is normal in $\GG$, we may define the quotient $\GG/_{\hspace{-0.2em}gp}\HH$ to be the co-equalizer of $\HH\rightrightarrows\GG$ in the category of affine group schemes, where one map is the natural inclusion and the other map is the trivial homomorphism.   One can show that $\GG/_{\hspace{-0.2em}gp}\HH$ is represented by $\Bbbk[\GG]^{\HH}$, see \cite[7.2.3]{C2}.

	\subsubsection{Factor groups} Given a morphism $\phi:\GG\to\GG'$ of algebraic groups, $\operatorname{ker}\phi$ is the subfunctor of $\GG$ given by $(\operatorname{ker}\phi)(A)=\operatorname{ker}(\phi(A):\GG(A)\to\GG'(A))$ for any commutative algebra $A$.  It is standard to show that $\operatorname{ker}\phi$ is representable and a normal subgroup of $\GG$, see \cite[7.2.1]{C2}.  	A factor group of $\GG$ is an algebraic group $\mathcal{Q}$ with a morphism of algebraic groups $\phi:G\to\mathcal{Q}$ such that $\Bbbk[\mathcal{Q}]\to\Bbbk[\GG]$ is injective.  By \cite[Thm.~6.3.1]{C2}, this is equivalent to $\phi$ being faithfully flat.   
		
	By \cite[Thm.~7.2.2]{C2}, there is a natural bijection between factor groups of $\GG$ and kernels of morphisms $\GG\to\GG'$.  Further, if $\phi:\GG\to\mathcal{Q}$ is a factor group, then $\mathcal{Q}$ is isomorphic to the sheafification in the fppf topology (see Section \ref{sec fppf topology}) of the functor $A\mapsto \GG(A)/\operatorname{ker}\phi(A)$.  

    We observe that if $\HH\trianglelefteq\GG$ is a normal subgroup, then $\GG/_{\hspace{-0.2em}gp}\HH$ is a factor group of $\GG$ under the canonical projection $\GG\to\GG/_{\hspace{-0.2em}gp}\HH$.  However, it is not clear that the kernel of $\GG\to\GG/_{gp}\HH$ is $\HH$, and it was asked in \cite{C2} if this is the case.  This would provide a bijection between normal subgroups and factor groups. We will answer this question in the affirmative in Section~\ref{MainSec}.

    \subsubsection{Image subgroup}  Let $\phi:\GG\to\GG'$ be a morphism of algebraic groups.  Define the image of $\phi$, written $\phi(\GG)$, to be the subfunctor of $\GG'$ given by the sheafification of the image subfunctor $A\mapsto\operatorname{Im}(\phi(A):\GG(A)\to\GG'(A))$.  An equivalent definition is to consider the image $B$ of the pullback morphism $\Bbbk[\GG']\to\Bbbk[\GG]$.  Then $B$ will be a Hopf subalgebra of $\Bbbk[\GG]$, and thus corresponds to a factor group of $\GG$, and we may define $\phi(\GG)=\Spec(B).$

	\begin{lemma}
		 Let $\phi:\GG\to\GG'$ be a morphism of algebraic groups. Then $\phi(\GG)\cong\GG/_{\hspace{-0.2em}gp}\operatorname{ker}\phi$ and is a subgroup of $\GG'$. 
	\end{lemma}
	
	\begin{proof}
		It is clear, by definition, that $\phi$ factors through $\GG/_{\hspace{-0.2em}gp}\operatorname{ker}\phi\to \GG'$.  Because the latter morphism has trivial kernel, it must be a monomorphism, and thus gives a subgroup of $\GG'$.  
	\end{proof}
	
	\subsection{A filtration}
	
	Recall that 
	we write $J=J_{\GG}=J_{\Bbbk[\GG]}$ for the ideal of $\Bbbk[\GG]$ generated by $\Bbbk[\GG]_{nil}$.  We have a filtration on $\Bbbk[\GG]$ by powers of $J$, which is finite since $\Bbbk[\GG]$ is finitely generated:
	\begin{equation}\label{eqfbp}
		0=J^{N+1}\subsetneq J^N\subsetneq\cdots \subsetneq J\subsetneq J^0=\Bbbk[\GG],
	\end{equation}
	for some $N\in\mathbb{N}$. Observe that (see (\ref{eqn nil}))
	\[
	\Delta(\Bbbk[\GG]_{nil})\subset(\Bbbk[\GG]\otimes\Bbbk[\GG])_{nil}\sub\Bbbk[\GG]_{nil}\otimes\Bbbk[\GG]+\Bbbk[\GG]\otimes\Bbbk[\GG]_{nil},
	\]
	which implies that 
	\[
	\Delta(J)\sub J\otimes\Bbbk[\GG]+\Bbbk[\GG]\otimes J,
	\]
	meaning that $J$ is a Hopf ideal.  Thus we have
	\begin{equation}\label{eqn J Delta}
		\Delta(J^k)\sub\sum\limits_{j}J^j\otimes J^{k-j}.   
	\end{equation}
	It follows that $\GG_0=\Spec\Bbbk[\GG]/J$ is an affine group scheme in $\Vec$.  Write $p_{\GG}:\Bbbk[\GG]\to\Bbbk[\GG_0]$ for the canonical quotient map.  Observe that the powers of $J$ are themselves left $\GG_0$-modules via
	\begin{equation}\label{eq module struc JJ^i}
		J^i\xto{\Delta}\sum\limits_{j}J^j\otimes J^{i-j}\xto{1\otimes p_\GG}J^i\otimes\Bbbk[\GG_0]. 
	\end{equation}

    \begin{remark}\label{rem:red}
        By replacing $J$ with the maximal nilpotent ideal, we can repeat the above, leading to the reduced subgroup (since $\Bbbk$ is algebraically closed) $\GG_{\mathrm{red}}\le \GG_{0}\le \GG$.
    \end{remark}

	\subsection{Connected components}\label{section conn components}
	As a finitely generated commutative algebra $A$ in $\cC$ is Noetherian, $|\Spec A|$ is a Noetherian topological space. In fact, by \cite[Cor. 5.3.2]{C2}, $|\Spec A|$ is naturally homeomorphic to $|\Spec A_{(0)}|$ and $A_{(0)}$ itself is Noetherian, by Proposition~\ref{Cor:AM}. Subsequently, the connected components of $\Spec A$ are the fibres of
	$$|\Spec A|\;\to\; |\Spec \pi(A)|,$$
	with $\pi(A)$ the maximal $\Bbbk$-\'etale subalgebra of $A$ (equivalently of $A_{(0)}$). We write $\pi_0(\Spec A)=\Spec \pi(A)$. As in the classical case, it follows that the component of our algebraic group $\GG$ containing the identity $\GG^\circ$ is a subgroup, leading to a short exact sequence
	$$1\to\GG^{\circ}\to \GG\to\pi_0(\GG)\to 1,$$
	and it follows that $\Bbbk[\GG^{\circ}]$ is $\Bbbk[\GG]e$, with $e$ the primitive idempotent not in the kernel of $\varepsilon$.

\begin{example}\label{ex:pi}
    If $\GG$ is finite, then it follows from Example~\ref{ex:finloc} that the composite homomorphism $\GG_{\mathrm{red}}\to\GG\to \pi_0(\GG)$ is a bijection.
\end{example}

	\subsection{Distribution Hopf algebra} 
We will devote a large portion of this section and the next to developing the fundamental theory of distribution algebras for affine group schemes in $\cC$. We expect this to be of independent interest. For the current paper, its relevance toward the main theorem is contained in the role of the distribution algebra towards proving the structural result for algebraic groups in Corollary~\ref{cor G0 to Gr surj}.

    \begin{definition}
    For an algebraic group $\GG$, define
    \[
    \Bbbk[\GG]^\circ:=\sum\limits_{I}(\Bbbk[\GG]/I)^*\,\subset\,\underline{\Hom}(\Bbbk[\GG],\mathbf{1}),
    \]
    where $I$ runs over all cofinite ideals of $\Bbbk[\GG]$.  Further, set
    \[
    \Gamma=\Gamma_{\GG}:=\sum\limits_{n\in\N}(\Bbbk[\GG]/\m_\varepsilon^n)^*\subset\Bbbk[\GG]^{\circ},
    \]
    where $\m_\varepsilon\sub\Bbbk[\GG]$ denotes the augmentation ideal.
\end{definition}
It is standard that $\Bbbk[\GG]^\circ$ has the structure of a cocommutative Hopf algebra in $\Cs$, and $\Gamma_{\GG}$ is a Hopf subalgebra with a unique simple subcoalgebra (\emph{irreducible} in the language of \cite{Mo}).  We consider the inclusion
\[G(\Bbbk)\subset \Bbbk[G]^{\circ},\;\, g\mapsto \delta_g.\]
Then the elements $\delta_{g}$ span the Hopf subalgebra $\Bbbk G(\Bbbk)\sub\Bbbk[G]^\circ$.

  In the following, if $A,B$ are Hopf algebras such that $A$ is a left $B$-module algebra, then we write $A\#B$ for the smash product of $A$ and $B$ (see \cite[Sec.~4.1]{Mo}).

\begin{lemma}\label{lemma decom k[G] circ}
    We have an isomorphism of Hopf algebras $\Bbbk[G]^\circ\cong\Gamma_{\GG}\#\Bbbk G(\Bbbk)$.  In particular, $\Gamma_{\GG}$ and $G(\Bbbk)$ generate $\Bbbk[G]^\circ$ as an algebra.
\end{lemma}

\begin{proof}
    The classical proof given in \cite[Sec.~5.6]{Mo} easily carries over to our setting.
\end{proof}

From the definition of $\Bbbk[\GG]^\circ$, we have a natural pairing
\begin{equation}\label{eqn pairing}
(-,-):\Bbbk[\GG]^\circ\otimes\Bbbk[\GG]\to\mathbf{1}.   
\end{equation}

\begin{prop}\label{prop nondegenerate pairing}
   The pairing (\ref{eqn pairing}) is nondegenerate.  Moreover, if $\GG=\GG^\circ$, then its restriction to $\Gamma_{\GG}\otimes\Bbbk[\GG]$ is nondegenerate. 
\end{prop}

For the proof of Proposition \ref{prop nondegenerate pairing}, we need a lemma. 

\begin{lemma}\label{lemma conn zero div}
    Let $\GG$ be connected.  Then $f\in\Bbbk[\GG]_{(0)}$ is a zero divisor if and only if $f$ is nilpotent.
\end{lemma}
\begin{proof}
Let $I=\sqrt{0}$ be the nilradical of $\Bbbk[\GG]$ (\cite[\S 2.2]{C2}), which is nilpotent since $\Bbbk[\GG]$ is noetherian.  Then $\Bbbk[\GG]/I$ is a reduced $\Bbbk$-algebra, and $\GG_{\mathrm{red}}=\Spec(\Bbbk[\GG]/I)$ is a subgroup of $\GG$ (see Remark \ref{rem:red}). Taking the associated graded with respect to the filtration generated by powers of $I$, we obtain
    \[
\gr\Bbbk[\GG]=\bigoplus\limits_{i}I^i/I^{i+1},
    \]
    and because $I$ is $\GG_{\mathrm{red}}$-stable, this will be a $\GG_{\mathrm{red}}$-equivariant $\Bbbk[\GG_{\mathrm{red}}]$-module, and thus by the fundamental theorem for Hopf modules it is free over $\Bbbk[\GG_{\mathrm{red}}]$.  If $f$ is a zero divisor, write $\ol{f}$ for the image of $f$ in $\Bbbk[\GG_{\mathrm{red}}]$.  If $\ol{f}=0$, then clearly it is nilpotent.  If $\ol{f}\neq0$, then because $\gr\Bbbk[\GG]$ is free over $\Bbbk[\GG_{\mathrm{red}}]$, $\ol{f}$ must be a zero divisor in $\Bbbk[\GG_{\mathrm{red}}]$.  However, for an algebraic group in $\Vec$, being reduced and connected implies it is integral, giving a contradiction.
\end{proof}

Recall from \cite[\S 4 and \S 5]{C2} that, by our assumptions on $\cC$, for every prime ideal $\mathfrak{p}$ in a commutative algebra $A$, there exists a local ring $A_{\mathfrak{p}}$, with the standard universal property, defined by inverting all elements in $A_{(0)}\backslash\mathfrak{p}_{(0)}$.

\begin{cor}\label{cor localisation embedding}
    Suppose that $\GG$ is connected.  Then for $\m\sub\Bbbk[\GG]$ a maximal ideal, the localisation map $\Bbbk[\GG]\to\Bbbk[\GG]_{\m}$ is a monomorphism.
\end{cor}
\begin{proof}
    This follows from Lemma~\ref{lemma conn zero div}.
\end{proof}

\begin{proof}[Proof of Proposition \ref{prop nondegenerate pairing}]
    It is clear that the pairing is nondegenerate in $\Bbbk[\GG]^\circ$, so let us show it is nondegenerate in $\Bbbk[\GG]$.  Suppose $X\sub\Bbbk[\GG]$ lies in the kernel of this pairing, then $e_gX$ also lies in the kernel, for any primitive idempotent $e_g$ in the maximal \'etale subalgebra of $\Bbbk[\GG]$ corresponding to some $g$ in the component group $\pi_0(\GG)=\GG/\GG^\circ$, see \S 3.3. Since $X\subset\oplus_g e_gX$ and each $e_g X$ lies in the kernel as well, we may thus assume that $X$ is contained in one component of $\Bbbk[\GG]$. Using left translation under $\GG(\Bbbk)$, we may assume this component is $\GG^\circ$, the connected component of the identity.  Therefore we may assume $\GG$ is connected.

If we write $\m_{\varepsilon}\sub\Bbbk[\GG]$ for the augmentation ideal, then since $\m_{\varepsilon}^n$ is cofinite for all $n\in\N$, it follows that $X\sub\bigcap\limits_{n\in\N}\m_{\varepsilon}^n$.  However Lemma \ref{lemma artin rees krull}(2) and Corollary  \ref{cor localisation embedding} imply that $\bigcap\limits_{n\in\N}\m_{\varepsilon}^n=0$, and this forces $X=0$ which finishes the proof.
\end{proof}

 As for general schemes, we say that an affine group scheme  $\HH$ is purely even if $\HH_0=\HH$. 
	\begin{lemma}\label{lemma gamma even GG even}
		$(\Gamma_{\GG})_{(0)}=\Gamma_{\GG}$ if and only if $\GG$ is purely even.
	\end{lemma}
	\begin{proof}
		The backward direction is clear, so let us suppose that $(\Gamma_{\GG})_{(0)}=\Gamma_{\GG}$.  By Proposition \ref{prop nondegenerate pairing}, the pairing $(-,-):\Gamma_{\GG}\otimes\Bbbk[\GG^\circ]\to\mathbf{1}$ is nondegenerate, which induces an inclusion $\Bbbk[\GG^\circ]\sub\Gamma_{\GG}^*$.  However the latter is purely even, so we obtain that the same is true for $\Bbbk[\GG^\circ]$.  The algebra $\Bbbk[\GG]$ is a direct sum of copies of $\Bbbk[\GG^\circ]$ as an object of $\mathscr{C}$, so $\Bbbk[\GG]$ is also purely even.
	\end{proof}

    \subsubsection{$\Gamma_{\GG}$ as a filtered Hopf algebra} Write $\Gamma=\Gamma_{\GG}$.  We call $\Gamma^{n}:=(\Bbbk[\GG]/\m_{\varepsilon}^{n+1})^*\sub\Gamma$ the distributions of degree $n$ (supported at $\varepsilon$).  Since $\Bbbk[\GG]$ is Noetherian, $\m_{\varepsilon}$ is finitely generated, so we obtain that $\Bbbk[\GG]/\m_{\varepsilon}^{n+1}$ is of finite length in $\mathscr{C}$, meaning the same is true of $\Gamma^{n}$.   This gives an ascending filtration
    \begin{equation}\label{eqn degree filt gama}
        \Bbbk\varepsilon=\Gamma^0\sub\Gamma^1\sub\Gamma^2\sub\cdots, \ \ \ \Gamma=\bigcup\limits_{n\in\N}\Gamma^n.
    \end{equation}
    Moreover we have $\Gamma^n\Gamma^m\sub\Gamma^{n+m}$, and thus $\Gamma^\bullet$ is a filtered algebra by subobjects of finite-length.
    
	It is now a standard fact that using the pairing $(-,-)$, $\Gamma$ inherits from $\GG$ the structure of a cocommutative Hopf algebra so that $(-,-)$ is a Hopf pairing.  We write $m_{\Gamma}$, resp.~$\Delta_{\Gamma}$, for its product, resp.~coproduct, $\eta_{\Gamma}$ and $\varepsilon_{\Gamma}$ for its unit and counit, and $\Gamma_{\GG}^+$ for its augmentation ideal.

	\begin{definition}
		Given a Hopf algebra $H$ in $\mathscr{C}$ with coproduct $\Delta$ and unit $\eta$, define $\operatorname{prim}H\sub H$ to be the kernel of 
		\[
		\Delta-\id\otimes\eta-\eta\otimes\id:H\to H\otimes H.
		\]
	\end{definition}
  
    We may view a Hopf algebra $H$ as a Lie algebra via the commutator.  A standard exercise shows that $\operatorname{prim}H$ is then a Lie subalgebra of $H$.
	\begin{definition}\label{defn Lie algebra}
		We define the Lie algebra of $\GG$, written $\operatorname{Lie}(\GG)=\g$, to be $\operatorname{prim}\Gamma_{\GG}$.
	\end{definition}
	
\begin{remark}
    	   We have $\g\sub\Gamma^1$, and a canonical decomposition  $\Gamma^{1}=\mathbf{1}\oplus\g$.
\end{remark}

	The filtration \eqref{eqfbp} on $\Bbbk[\GG]$ by powers of $J$ induces a finite filtration on $\Gamma$:
	\begin{equation}\label{eqn finite filt gamma}
	0\subsetneq\Gamma_0=\Gamma_{\GG_0}\subsetneq\Gamma_1\subsetneq\cdots\subsetneq\Gamma_{N}=\Gamma,
	\end{equation}
	where $\Gamma_i$ is the kernel of
	\[
	\Gamma\otimes J^{i+1}\to \mathbf{1}.
	\]
	By (\ref{eqn J Delta}) we have $\Gamma_i\Gamma_j\sub\Gamma_{i+j}$ and, furthermore,
	\begin{equation}\label{eqn Gamma delta}
    \Delta(\Gamma_k)\sub\sum\limits_j\Gamma_j\otimes\Gamma_{k-j}.    
	\end{equation}
	
	\begin{remark}
		We have now constructed two filtrations on $\Gamma$, denoted by $\Gamma^\bullet$ and $\Gamma_\bullet$, both giving $\Gamma$ the structure of a filtered algebra. The first one is a direct analogue of the classical filtration in $\mathrm{Vec}$, while the second one is only relevant for $\cC\not=\mathrm{Vec}$. We will use both filtrations, but $\operatorname{gr}\Gamma$ will always mean the associated graded of $\Gamma$ with respect to the finite filtration $\Gamma_{\bullet}$.
	\end{remark}
	
	Observe that if $\phi:\HH\to\GG$ is a morphism of algebraic groups, then  $\epsilon_{\GG}=\epsilon_{\HH}\circ \phi$.  Therefore, there is an induced morphism of distributions $\Gamma_\phi:\Gamma_{\HH}\to\Gamma_{\GG}$, and one may check that this is a morphism of Hopf algebras.

	\begin{lemma}\label{lemma ses dist algs}
		If $1\to\NN\to\GG\to\mathcal{Q}\to1$ is a short exact sequence of algebraic groups, then we have a short exact sequence:
		\[
		0\to(\Gamma_{\NN}^+)\to\Gamma_{\GG}\to\Gamma_{\mathcal{Q}}\to0.
		\]where $(\Gamma_{\NN}^+)$ is the ideal of $\Gamma_{\GG}$ generated by the augmentation ideal of $\Gamma_{\NN}$.
	\end{lemma}
	\begin{proof}
	For a proof of the fact that $\Gamma_{\GG}\to\Gamma_{\QQ}$ is an epimorphism, we note that the proof of \cite[Chpt.~I, Prop.~7.6]{J} (including the referenced result \cite[Chpt.~I, \S3, Prop.~9]{Bour}) carries over to our setting, where we use that $\GG\to\QQ$ is flat.  
    
    To prove exactness in the middle, observe that we have an exact sequence
		$$0\to \Bbbk[\mathcal{Q}]\to\Bbbk[\GG]\xrightarrow{(1\otimes q)\circ\Delta} \Bbbk[\GG]\otimes \Bbbk[\NN]/\Bbbk 1,$$
        where we denote by $q$ the composite of the defining Hopf algebra quotient $\Bbbk[\mathcal{G}]\to \Bbbk[\mathcal{N}]$ with the map $\Bbbk[\mathcal{N}]\to \Bbbk[\mathcal{N}]/\Bbbk 1$ quotienting out the algebra unit. Appropriately dualising this sequence gives middle exactness, as desired.
	\end{proof}

	\begin{cor}\label{cor iso alg gps criteria}
		For a morphism $\phi:\GG\to\GG'$ of algebraic groups, the following are equivalent:
		\begin{enumerate}
			\item $\phi$ is an isomorphism,
			\item $\phi_0:\GG_0\to\GG_0'$ and $\Gamma_\phi:\Gamma_{\GG}\to\Gamma_{\GG'}$ are isomorphisms, and
			\item $\phi(\Bbbk):\GG(\Bbbk)\to\GG'(\Bbbk)$ and $\Gamma_\phi$ are isomorphisms. 
		\end{enumerate}
	\end{cor}
	
	\begin{proof}
		Clearly $(1)\Rightarrow(2)\Rightarrow(3)$, since $\GG(\Bbbk)=\GG_0(\Bbbk)$.  To show (3) $\Rightarrow$ (1), write $I$ for the kernel of the pullback $\phi^*:\Bbbk[\GG']\to\Bbbk[\GG]$.  We have an induced morphism $\Bbbk[\GG]^\circ\to\Bbbk[\GG']^\circ$, and by Lemma \ref{lemma decom k[G] circ} and our assumptions, it will be an isomorphism.  Thus $\Bbbk[\GG']^\circ$ must pair with $I$ to 0 under (\ref{eqn pairing}), implying that $I=0$ by Proposition \ref{prop nondegenerate pairing}. It follows that $\phi^*$ is a monomorphism, and thus $\phi$ is a quotient morphism of algebraic groups.  If we write $\NN$ for its kernel,  then $\NN(\Bbbk)=\operatorname{ker}(\phi(\Bbbk))$ is trivial.  This implies $\NN$ is infinitesimal, so that $\Gamma_{\NN}=\Bbbk[\NN]^*$ (see $\S$\ref{section infinitesimal schemes}).  However, by Lemma \ref{lemma ses dist algs} we see that $\Gamma_{\NN}=1$, and thus $\Bbbk[\NN]=\mathbf{1}$, so that $\NN$ is trivial, and $\phi$ is an isomorphism. 
	\end{proof}

	\subsection{Semi-direct products}
	In the following, we refer to \cite{Mol} for the notion of a smash (co)product of Hopf algebras.  
	
	Let $\NN$ and $\HH$ be algebraic groups, and suppose that we have a left action $a:\HH\times\NN\to\NN$ of $\HH$ on $\NN$ which respects the group structure on $\NN$.  Then we may form the semi-direct product $\HH\ltimes\NN$, which as a scheme is simply the product $\HH\times\NN$, and the group structure is given on points by $(\HH\ltimes\NN)(A)=\HH(A)\ltimes\NN(A)$.  In particular, as commutative algebras we have $\Bbbk[\HH\ltimes\NN]\cong \Bbbk[\HH]\otimes\Bbbk[\NN]$, and as a Hopf algebra $\Bbbk[\HH\ltimes\NN]$ is the smash coproduct of $\Bbbk[\NN]$ by $\Bbbk[\HH]$. 
	
	\begin{lemma}\label{lemma semidirect prod dist alg}
		We have that $\Gamma_{\HH\ltimes\NN}$ is the smash product of $\Gamma_{\NN}$ by $\Gamma_{\HH}$.  In particular, the multiplication map $\Gamma_{\NN}\otimes\Gamma_{\HH}\to\Gamma_{\HH\ltimes\NN}$ is an isomorphism of coalgebras in $\mathscr{C}$.
	\end{lemma}
	
	\begin{proof}
		This follows from the natural generalization of \cite[Thm.~5.1]{Mol} to $\mathscr{C}$.
	\end{proof}

\begin{example}
    If $\GG$ is finite (meaning $\Bbbk[\GG]$ is compact), then it follows from Example~\ref{ex:pi} that $\GG\cong \pi_0(\GG)\ltimes \GG^{\circ}$ and that $\GG^{\circ}$ is infinitesimal. So each finite group scheme is a semidirect product of an abstract finite group, and an infinitesimal group scheme.
\end{example} 
	
	\subsection{The group $\operatorname{gr}\GG$} Using the finite filtrations (\ref{eqfbp}) and (\ref{eqn finite filt gamma}) on $\Bbbk[\GG]$ and $\Gamma_{\GG}$, we may form the associated graded objects $\gr\Bbbk[\GG]$ and $\gr\Gamma_{\GG}$. By (\ref{eqn J Delta}) and (\ref{eqn Gamma delta}), these will once again be Hopf algebras, and we have an induced pairing:
	\[
	\gr\Bbbk[\GG]\otimes\gr\Gamma_{\GG} \to \mathbf{1}.
	\]
	Notice that $\gr\Bbbk[\GG]$ is generated, as an algebra, by $\Bbbk[\GG]/J=\Bbbk[\GG_0]$ and $J/J^2$.  The former is a finitely generated algebra, and the latter is a finitely generated module over $\Bbbk[\GG_0]$, implying that $\gr\Bbbk[\GG]$ is finitely generated.  Write $\operatorname{gr}\GG$ for the algebraic group represented by $\gr\Bbbk[\GG]$. Since $\operatorname{gr}\m_{\varepsilon}=\m_{\operatorname{gr}\varepsilon}$, the pairings produce a factorisation
	\[
	\xymatrix{
		\gr\Gamma_{\GG}\ar[rr]\ar@{-->}[rd]_{\psi}&&\Bbbk[\gr\GG]^\ast\\
		&\Gamma_{\gr\GG}\ar@{^{(}->}[ru]
	}
	\]
	where $\psi$ is a morphism of Hopf algebras. 

    Any filtration on an object $V$ in $\cC$ induces a filtration on any subobject $U\subset V$ and we have a natural inclusion $\gr U\subset\gr V$. We apply this in particular to the subobject $\Gamma^n\subset \Gamma$ from the degree filtration (\ref{eqn degree filt gama}) on $\Gamma$, leading to $\gr \Gamma^n\subset \gr \Gamma$.
	
	\begin{lemma}\label{lemma gr dist algs commutes}
		The map $\psi$ is an isomorphism of Hopf algebras and we have $\psi(\operatorname{gr}\Gamma_{\GG}^{n})\sub \Gamma_{\operatorname{gr}\GG}^{n}$.  
	\end{lemma}

    \begin{proof}
		Suppose that $X\sub \Gamma_{i+1}/\Gamma_i$ is non-zero and of finite length.  Then there exists a finite-length subobject $Y\sub \Gamma_{i+1}$ such that $X\cong Y/\Gamma_{i}\cap Y$.  By nondegeneracy in $\Gamma$ of the pairing $\Gamma\otimes\Bbbk[\GG]\to\Bbbk$, there exists a compact subobject $U\sub J^{i+1}$ such that the pairing $Y\otimes U\onto X\otimes U\to \mathbf{1}$ is non-zero. Thus if $V=U/U\cap J^{i+2}$, we have that $X\otimes V\to\mathbf{1}$ is non-zero. 
		It follows that $\psi$ is a monomorphism.  
		
		To prove surjectivity, observe that $\gr\m_{\varepsilon}^n=\m_{\gr\epsilon}^n$, and hence $\psi(\gr\Gamma^{n})\sub\Gamma_{\gr \GG}^{n}$.  Since $\gr(\Bbbk[\GG]/\m_{\varepsilon}^n)$ and $\gr(\Bbbk[\GG])/\m_{\gr\varepsilon}^n$ have the same length, the same is true of $\Gamma_{\operatorname{gr}\GG}^n$ and $\operatorname{gr}\Gamma_{\GG}^n$, and thus $\psi$ is an  epimorphism, and therefore an isomorphism.
	\end{proof}
	Recall that $\g$ denotes the Lie algebra of $\GG$ (Definition \ref{defn Lie algebra}).
	\begin{cor}\label{cor gr Lie algebra}
		We have $\operatorname{Lie}\operatorname{gr}\GG=\g_{(0)}\ltimes\g/\g_{(0)}$, where $\g/\g_{(0)}$ is an abelian ideal of $\operatorname{Lie}\operatorname{gr}\GG$.   
	\end{cor}
	\begin{proof}
		Indeed, $\Gamma^{1}=\mathbf{1}\oplus\g$, $\g\cap\Gamma_{\GG_0}=\g_{(0)}$, and $\g\sub\Gamma_1$ (since $J^2\sub\m^2)$.  Thus $\operatorname{gr}\Gamma^1=\mathbf{1}\oplus\g_{(0)}\oplus\g/\g_{(0)}$, and the Lie bracket structure is easy to check.
	\end{proof}

	We say that an algebraic group $\NN$ is infinitesimal if it is infinitesimal as a scheme. In particular, by Lemma~\ref{lemma inftsml is local fin length}, $\Gamma_{\NN}$ then has finite length and $\Bbbk[\NN]\cong\Gamma_{\NN}^*$ as Hopf algebras.
	
	\begin{lemma}\label{lemma grG iso}
		We have an isomorphism:
		\[
		\operatorname{gr}\GG\cong \NN\rtimes \GG_0,
		\]
		where $\NN$ is a commutative, infinitesimal, algebraic group with $\operatorname{Lie}(\NN)=\g/\g_{(0)}$.  Further, multiplication induces a coalgebra isomorphism:
		\[
		\Gamma_{\GG_0}\otimes\Gamma_{\NN}\cong\Gamma_{\operatorname{gr}\GG}.
		\]
	\end{lemma}
	
	\begin{proof}
		We have a splitting of Hopf algebras $\Bbbk[\GG_0]\to \Bbbk[\operatorname{gr}\GG]$ of the natural surjection $\Bbbk[\operatorname{gr}\GG]\to\Bbbk[\GG_0]$, so we define the algebraic group $\NN$ by $\Bbbk[\NN]:=\Bbbk[\gr \GG]/(\Bbbk^+[\GG_0])$.  In other words, we have a split short exact sequence
        \[
        1\to \NN\to\gr\GG\to\GG_0\to1.
        \]
         In particular, we have an isomorphism $\operatorname{gr}\GG\cong\GG_0\ltimes\NN$.  Since $\gr\Bbbk[\GG]$ is a finitely generated $\Bbbk[\GG_0]$-module, $\Bbbk[\NN]$ has finite length. Since it is also an $\N$-graded algebra with degree 0 part isomorphic to $\mathbf{1}$, the group $\NN$ must be infinitesimal. To show that $\NN$ is commutative, we need to prove that $\Bbbk[\NN]$ is cocommutative.  Since $\Bbbk[\NN]$ is generated in degree one as an $\N$-graded algebra, it suffices to show that its degree one part is primitive.  But this follows from the fact that the coproduct on $\Bbbk[\NN]$ preserves the grading, and the degree 0 part is isomorphic to $\mathbf{1}$.  To prove the statement on Lie algebras, we use Corollary \ref{cor gr Lie algebra}.  Finally, to prove the isomorphism of coalgebras for distribution algebras, we may apply Lemma \ref{lemma semidirect prod dist alg}.
	\end{proof}

	\begin{cor}\label{cor fg gamma}
		$\Gamma_{\GG}$ is a finitely generated $\Gamma_{\GG_0}$-module under left multiplication.  
	\end{cor}
	\begin{proof}
		By Lemmas \ref{lemma gr dist algs commutes} and \ref{lemma grG iso}, we have $\gr\Gamma_{\GG}=\Gamma_{\operatorname{gr}\GG}=\Gamma_{\GG_0}\otimes\Gamma_\NN$ where $\Gamma_{\NN}$ is of finite-length.  Let $X\sub\Gamma_{\GG}$ be a compact subobject such that $\gr X$ contains $\Gamma_\NN$.  Then the multiplication map $\Gamma_{\GG_0}\otimes X\to \Gamma_{\GG}$ becomes an {epimorphism} after taking the associated graded, and thus is an {epimorphism}, implying finite generation. 
	\end{proof}

	\begin{example}
		If $\XX$ is the additive group scheme associated to a compact object $X^*$ in $\mathscr{C}$, then we have a short exact sequence 
		\[
		0\to X_{nil}\to X\to X/X_{nil}\to0.
		\]
		Now \cite[Thm.~3.3.2]{CEO1} implies that $\operatorname{gr}SX\cong  S(X/X_{nil})\otimes R$, where $R:=\operatorname{Im}(S(X_{nil})\to S(X))$.  Thus $\operatorname{Spec}S(X/X_{nil})\cong \XX_0$ and $\Spec R\cong \NN$ in the notation of Lemma \ref{lemma grG iso}.
	\end{example}
	
	\begin{example}
		Suppose that $\mathscr{C}=\operatorname{Ver}_p$, and $\GG$ is an algebraic group corresponding to a Harish-Chandra pair $(\GG_0,\mathfrak{g})$ as introduced in \cite{V1}.  Then $\operatorname{gr}\GG$ corresponds to the Harish-Chandra pair $(\GG_0,\g_{(0)}\ltimes\g_{\neq0})$, where $\g=\g_{(0)}\oplus\g_{\neq0}$.
	\end{example}
	
	\begin{example}
		Suppose that $\mathscr{C}=\operatorname{Ver}_4^+$ and let $P$ denote the unique projective indecomposable.  Then $\operatorname{gr}\GG\LL(P)\cong \G_m\times\G_a\times \alpha_2^2$, where $\alpha_2$ is the first Frobenius kernel of $\G_a$.
	\end{example}
	\subsection{A presentation for $\Bbbk[\GG]$}
	
	Consider the object $\ul{\Hom}_{\Gamma_{\GG_0}}(\Gamma,\Bbbk[\GG_0])$, where $\Gamma$ is a left $\Gamma_{\GG_0}$ module under left multiplication, and $\Bbbk[\GG_0]$ is a left $\Gamma_{\GG_0}$-module under infinitesimal right translation, i.e.
	\begin{equation}\label{eqn right trans}
		\Gamma_{\GG_0}\otimes\Bbbk[\GG_0]\xto{\sim}\Bbbk[\GG_0]\otimes \Gamma_{\GG_0}\xto{\Delta_{\GG_0}\otimes 1} \Bbbk[\GG_0]\otimes\Bbbk[\GG_0]\otimes\Gamma_{\GG_0}\xto{1\otimes (-,-)}\Bbbk[\GG_0]
	\end{equation}
	where we have used the braiding on $\Vec$ in the first isomorphism.
	Let us give $\ul{\Hom}_{\Gamma_{\GG_0}}(\Gamma,\Bbbk[\GG_0])$ the structure of a commutative algebra.  First observe that we have an isomorphism
	\[
	\ul{\Hom}_{\Gamma_{\GG_0}}(\Gamma,\Bbbk[\GG_0])\otimes \ul{\Hom}_{\Gamma_{\GG_0}}(\Gamma,\Bbbk[\GG_0])\to \ul{\Hom}_{\Gamma_{\GG_0}\otimes\Gamma_{\GG_0}}(\Gamma\otimes\Gamma,\Bbbk[\GG_0]\otimes\Bbbk[\GG_0]),
	\]
	which can be proven using the $\otimes-\ul{\Hom}$ adjunction (\ref{eqn tensor hom ring}).  The multiplication is described as follows:
	\[
	m_{\Hom}=\ul{\Hom}(\Delta_{\Gamma},m_{\GG_0})
	\]  
	We set $\epsilon_{\Hom}=\ul{\Hom}(\eta_{\Gamma},\epsilon_{\GG_0})$ for the unit.  Here we use that $\ul{\Hom}$ is a bifunctor to define our morphisms, and that $\ul{\Hom}_{\Gamma_{\GG_0}}$ and $\ul{\Hom}_{\Gamma_{\GG_0}\otimes\Gamma_{\GG_0}}$ are subfunctors of $\ul{\Hom}$ on their respective categories of modules.

	In the following, we view $\Bbbk[\GG]$ as a left $\Gamma_{\GG}$-module under the action by infinitesimal right translation (as in (\ref{eqn right trans})).  We view $\ul{\Hom}_{\Gamma_{\GG_0}}(\Gamma_{\GG},\Bbbk[\GG_0])$ as a left $\Gamma_{\GG}$-module via the action of $\Gamma_{\GG}$ on itself by right multiplication.  
	
	The following presentation for $\Bbbk[\GG]$ was first observed in the super case over the complex numbers by Koszul \cite{K}.  For more general proofs, see \cite{MS} for the super case and \cite{V1} for the case of $\operatorname{Ver}_p$ of the following theorem.
	
	\begin{thm}\label{thm presentation k[G]}
		We have a natural isomorphism of algebras and $\Gamma$-modules
		\[
		\Phi:\Bbbk[\GG]\to\ul{\Hom}_{\Gamma_{\GG_0}}(\Gamma_{\GG},\Bbbk[\GG_0]),
		\]
		induced, via adjunction, by the natural quotient map $\Bbbk[\GG]\to\Bbbk[\GG_0]$ of left $\Gamma_{\GG_0}$-modules.
	\end{thm}
	
	\begin{proof}
		The finite cofiltration $\Gamma/\Gamma_\bullet$ on $\Gamma$ induces a descending filtration $\ul{\Hom}_{\Gamma_{0}}(\Gamma/\Gamma_\bullet,\Bbbk[\GG_0])$ on $\ul{\Hom}_{\Gamma_{\GG_0}}(\Gamma,\Bbbk[\GG_0])$, and this filtration respects the algebra structure.  We claim that $\Phi(J^i)\sub\ul{\Hom}_{\Gamma_{\GG_0}}(\Gamma/\Gamma_{i-1},\Bbbk[\GG_0])$.   Indeed, it is clear that $J$ lands in $\ul{\Hom}(\Gamma/\Gamma_{\GG_0},\Bbbk[\GG_0])$, and from this the claim follows.
		
		Now we have a morphism of filtered algebras; taking the associated graded, we obtain the morphism
		\[
		\operatorname{gr}\Bbbk[\GG]\to \ul{\Hom}_{\Gamma_{\GG_0}}(\operatorname{gr}\Gamma,\Bbbk[\GG_0]),
		\]
		or by Lemma \ref{lemma grG iso},
		\[
		\Bbbk[\NN]\otimes\Bbbk[\GG_0]\to\ul{\Hom}_{\Gamma_{\GG_0}}(\Gamma_{\GG_0}\otimes\Gamma_{\NN},\Bbbk[\GG_0])\cong \ul{\Hom}_{\Gamma_{\GG_0}}(\Gamma_{\GG_0},\Bbbk[\GG_0])\otimes(\Gamma_{\NN})^*.
		\]
		Since $\NN$ is infinitesimal and algebraic, $(\Gamma_{\NN})^*\cong\Bbbk[\NN]$.  From this, the isomorphism is not hard to show. 
	\end{proof}
	
	Now if $\UU\sub \GG$ is an open subscheme, then $\Gamma_{\GG_0}$ acts naturally on the left on $\Bbbk[\UU_0]$ by infinitesimal right translation, and so we may construct a presheaf of $\mathscr{C}$-algebras on $|\GG|$ by the association:
	\[
	\UU\mapsto \ul{\Hom}_{\Gamma_{\GG_0}}(\Gamma,\Bbbk[\UU_0]).
	\]
	
	\begin{prop}\label{lemma str sheaf desc}
		The above assignment is isomorphic to the structure sheaf of $\GG$.
	\end{prop}
	\begin{proof}
		Let us view $\ul{\Hom}_{\Gamma_{\GG_0}}(\Gamma,-)$ as functor from $\Gamma_{\GG_0}$-modules in $\Vec$ to $\mathscr{C}$.  Then we claim that this functor is exact.  Indeed, the cofiltration $\Gamma/\Gamma_{\bullet}$ on $\Gamma$ induces a filtration on this functor. Since we showed that $\gr \Gamma \simeq \Gamma_{\gr \GG}\simeq \Gamma_{\GG_0}\otimes \Gamma_{\mathcal{N}}$ (see Lemma~\ref{lemma grG iso}) the layers are given by $\ul{\Hom}_{\Gamma_{\GG_0}}(\Gamma_{\GG_0}\otimes X_i,-)$ where $X_i\in\mathscr{C}$.  Since these layers are clearly exact functors, the whole functor is as well.  
		
		It follows that the functor $|\UU|\mapsto \ul{\Hom}_{\Gamma_{\GG_0}}(\Gamma,\Bbbk[\UU_0])$ defines a sheaf on $|\GG|$.  This sheaf admits a natural map from $\OO_{\GG}$, and by using a parallel argument to that in Theorem \ref{thm presentation k[G]}, it is an isomorphism.
	\end{proof}
	
	\section{Frobenius twist}\label{FrobTw}
	
	Write $\mathscr{C}^{(1)}$ for the Frobenius twist of $\mathscr{C}$, which is equivalent to $\mathscr{C}$ as an additive symmetric monoidal category,  but we alter the $\Bbbk$-linear structure by letting $\lambda$ act by $\lambda^p$ for $\lambda\in\Bbbk$.  Recall the Frobenius functor $\operatorname{Fr}:\mathscr{C}\to\mathscr{C}^{(1)}\boxtimes\operatorname{Ver}_p$ is a $\Bbbk$-linear symmetric monoidal functor defined by the composition of functors:
	\[
	\mathscr{C}\xto{\otimes p} \operatorname{Rep}_{\mathscr{C}}C_p\to \mathscr{C}^{(1)}\boxtimes\operatorname{Ver}_p.
	\]
	The first functor is given by $X\mapsto X^{\otimes p}$, and the second functor is described by certain cohomological functors (see \cite{C1,EO}).  {An equivalent description is given by} applying the semisimplification functor from $\Rep C_p$ to $\mathrm{Ver}_p$ for $\Rep_{\mathscr{C}}C_p$ viewed as a category of comodules in $\Rep C_p$ over some coalgebra, see \cite{CEO1} for more details.
	
	Now let $A$ be a commutative algebra in $\mathscr{C}$. View $A$ as an object of $\operatorname{Rep}_{C_p}\mathscr{C}$ with trivial $C_p$ action.  Then multiplication $A^{\otimes p}\to A$ is equivariant with respect to the $C_p$-action, and thus gives rise to a natural map:
	\[
	\phi_A:\operatorname{Fr}(A)\to A.
	\]
	Moreover, $\phi_A$ is an algebra morphism in $\mathscr{C}^{(1)}\boxtimes\operatorname{Ver}_p$, where $A=A\boxtimes \mathbf{1}$.  Write $A^{[1]}:=\operatorname{im}\phi_A$, so that $A^{[1]}$ is a subalgebra of $A$ inside of $\mathscr{C}$. 
    
    \begin{definition}
        We call $A^{[1]}$ the \emph{Frobenius twist} of $A$. For $r\in\N$, define $A^{[r]}$ inductively by $A^{[r]}:=(A^{[r-1]})^{[1]}$. 
    \end{definition}   For an object $X$ in a tensor category and $n\in\mathbb{N}$, we denote by $(X^{\otimes n})^{S_n}$ its $n$-th divided power, {\it i.e.} the maximal $S_n$-invariant subobject of $X^{\otimes n}$.
	
	\begin{lemma}
		We have $A^{[1]}=\operatorname{im}\left((A^{\otimes p})^{S_p}\to A\right)$.
	\end{lemma}
	
	\begin{proof}
		We may view $A^{\otimes p}$ as an object of either $\mathscr{C'}:=\mathscr{C}^{(1)}\boxtimes\operatorname{Rep}C_p$ or $\mathscr{C}'':=\mathscr{C}^{(1)}\boxtimes\operatorname{Rep}S_p$.  Notice that we have a natural tensor functor $\mathscr{C}''\to\mathscr{C}'$, and because $C_p$ is a Sylow subgroup of $S_p$, it induces a tensor functor fitting into the following diagram with $\operatorname{Fr}^{en}$ the enriched Frobenius functor from \cite{CEO2}:
		\[
		\xymatrix{
			& \mathscr{C}\ar[dr]\ar[dl] \ar@/^5pc/[ddr]^{\operatorname{Fr}} \ar@/_5pc/[ddl]_{\operatorname{Fr}^{en}}& \\
			\mathscr{C}''\ar[rr]\ar[d] & & \mathscr{C}'\ar[d]\\
			\mathscr{C}^{(1)}\boxtimes(\operatorname{Rep}S_p)^{ss}\ar@{-->}[rr] & & \mathscr{C}^{(1)}\boxtimes\operatorname{Ver}_p.
		}
		\]
		Here we write $(\operatorname{Rep}S_p)^{ss}$ for the semisimplification of $\Rep S_p$.  Because $A$ has a trivial action of $S_p$, upon semisimplification the morphism $A^{\otimes p}\to A$ will have image given by the image of $(A^{\otimes p})^{S_p}$ inside $A$ (or equivalently, the image of $\operatorname{Fr}_+(A)\to A$, with $\operatorname{Fr}_+$ the `identity component' of $\operatorname{Fr}^{en}$ as discussed in detail in \cite{C1}), and so we are done. 
	\end{proof}

	\begin{remark}
		For $\mathscr{C}=\mathrm{Vec}$, the extension of scalars $A^{(1)}=\Bbbk\otimes_{\Bbbk} A$ of $A$ along the Frobenius homomorphism on $\Bbbk$ is usually referred to as the Frobenius twist of $A$. There is the $\Bbbk$-linear $p$-th power map
		$$A^{(1)}\to A,\quad 1\otimes a\mapsto a^p.$$
		The algebra $A^{[1]}$ is the image of the above morphism. However, $A^{(1)}\twoheadrightarrow A^{[1]}$ is only an isomorphism if $A$ is reduced. This is why we use the notation $A^{[1]}$ in place of $A^{(1)}$ for our notion of Frobenius twist.
	\end{remark}

	\subsection{Properties of Frobenius twist}
	As above, $A$ denotes a commutative algebra in $\Cs$.
	
	\begin{lemma}\label{lemma A finite over Fr(A)}
		Suppose that $A$ is finitely generated.
		\begin{enumerate}
			\item If $A\neq\mathbf{1}$, then $A^{[1]}$ is a proper subalgebra of $A$.
			\item $A$ is finite over $A^{[1]}$. 
			\item $A^{[1]}$ is a finitely generated algebra.
			\item The kernel of the natural morphism $\ol{A^{[1]}}\to\bar{A}$ is a nilpotent ideal.  Further, the image contains the subalgebra $(\bar{A})^{[r]}=\{a^{p^r}|a\in \bar{A}\}$ for some $r\in\N$. 
			\item The morphism $\operatorname{Spec} A\to\operatorname{Spec}A^{[1]}$ is a universal homeomorphism.
		\end{enumerate}
	\end{lemma}
	\begin{proof}
		Let $\m$ denote a nonzero maximal ideal of $A$, so that we have a splitting $A=\mathbf{1}\oplus\m$.  Then we see that $A^{[1]}\sub\mathbf{1}\oplus\m^p$.  Since $\m^p\neq\m$ unless $\m=0$ (Thm.~6.1.5 of \cite{C2}), $A^{[1]}$ must be a proper subalgebra.
		
		For (2), choose generators $f_1,\dots,f_k$ of $\bar{A}$ over $\Bbbk$.  Then by GR, there exist $n_1,\dots,n_k\in\N$ such that $f_1^{n_1},\dots,f_k^{n_k}$ lift to $A$.  Write $B=B_{(0)}\sub A$ for the subalgebra of $A$ generated by these lifts, so that $B$ splits off $A$ as an object in $\mathscr{C}$.  Then $B^{[1]}$ is generated by $f_1^{n_1p},\cdots,f_k^{n_kp}$, meaning that $B^{[1]}\to A^{[1]}\to A\to \bar{A}$ is a finite morphism by construction.  We conclude by Lemma \ref{lemma finite subalg}.
		
		Part (3) now follows from (2) and Lemma \ref{lemma atiyah-mac}.
		
		Part (4) is clear.  For (5), we use (4)  and Lemma \ref{lemma univ homeo stacks}  to show that $\operatorname{Spec}\bar{A}\to\operatorname{Spec}\ol{A^{[1]}}$ is a universal homeomorphism.  The result now follows from (2) of Lemma \ref{lemma even embedding univ homeo}.
	\end{proof}
	    
	\begin{lemma}\label{lemma Fr and ring maps localization}  Let $A,B$ be commutative algebras.
		\begin{enumerate}
			\item We have a canonical isomorphism of algebras $(A\otimes B)^{[1]}\cong A^{[1]}\otimes B^{[1]}$.
			\item If $A\to B$ is injective, then so is $A^{[1]}\to B^{[1]}$.  
			\item If $A\to B$ is finite and { $A$ is finitely generated}, then $A^{[1]}\to B^{[1]}$ is finite.
			\item If $f\in A_{(0)}$, then $f^p\in A^{[1]}$, and
			\[
			(A^{[1]})_{f^p}\cong (A_{f^p})^{[1]}.
			\]
		\end{enumerate}    
	\end{lemma}
	\begin{proof}
		Part (1) follows from the monoidality of the Frobenius functor.  Part (2) is immediate.  For (3), we use (2) of Lemma \ref{lemma A finite over Fr(A)} to obtain that $A^{[1]}\to A\to B$ is a finite morphism.  Since $B^{[1]}$ is an $A^{[1]}$-submodule of $B$, and $A^{[1]}$ is Noetherian by (3) of Lemma \ref{lemma A finite over Fr(A)}, we conclude that $B^{[1]}$ is also finite over $A^{[1]}$.
		
		For (4), we see that:
		\begin{eqnarray*}
			(A^{[1]})_{f^p}& = &\lim\limits_{\rightarrow}(A^{[1]}\xto{f^p}A^{[1]}\xto{f^p}\cdots)\\
			& = &\lim\limits_{\to}\left(\operatorname{im}((A^{\otimes p})^{S_p}\to A)\xto{f^{p}}\operatorname{im}((A^{\otimes p})^{S_p}\to A)\xto{f^{p}}\cdots\right)\\
			& = &\operatorname{im}(((A^{\otimes p})^{S_p})_{f^{\otimes p}}\to A_{f^p}).
		\end{eqnarray*}
		Here we have used the exactness of filtered colimits.  It remains to observe that 
		\[
		((A^{\otimes p})^{S_p})_{f^{\otimes p}}\cong (A_f^{\otimes p})^{S_p}\cong (A_{f^p}^{\otimes p})^{S_p},
		\]
        again by exactness of localisation.
	\end{proof}

    \begin{example}\label{example verp frob twist}
        If $\Cs=\operatorname{Ver}_p$, for any commutative algebra $A$ we may choose a surjection $S(X)\onto A$ for some $X\in\operatorname{Ver}_p$.  Because the Frobenius functor $\operatorname{Fr}$ is exact on $\operatorname{Ver}_p$, we obtain a surjection $S(X)^{[1]}\onto A^{[1]}$.  We have a decomposition $X=\bigoplus\limits_i L_i^{\oplus I_i}$ for some index sets $I_i$, then 
        \[
        S(X)^{[1]}\cong \bigotimes_{i}\bigotimes_{I_i}S(L_i)^{[1]}\cong \bigotimes_{I_0}S(\mathbf{1})^{[1]},
        \]
        where we take restricted tensor product.  Here we used (1), Lemma \ref{lemma Fr and ring maps localization} and exactness of $\operatorname{Fr}$ to distribute the Frobenius twist over the tensor product.  In particular, we see that $S(X)^{[1]}$ is a $\Bbbk$-algebra, implying that $A^{[1]}$ is too.
    \end{example}

\begin{remark}
    Frobenius twists and kernels of algebraic groups over $\operatorname{Ver}_p$ (using an equivalent definition) were studied and used in \cite{Ka} to prove the Steinberg tensor product theorem in this setting. 
\end{remark}

	\begin{remark}\label{rmk caution twist surj}
		It is \emph{not} true that Frobenius twist preserves surjectivity or pushouts of commutative algebras.  Indeed, let $\mathscr{C}=\operatorname{Ver}_4^+$ and $A=S(P)=\Bbbk[x,y]/x^2$ where $P=\Bbbk\langle x,y\rangle$ is the unique indecomposable projective.  Then we have $A^{[1]}=\Bbbk[y^4]$, but $\bar{A}=\Bbbk[y]$ and thus $\bar{A}^{[1]}=\Bbbk[y^2]$.  Further, we see that 
		\[
		\Bbbk[y^2]=(\bar{A}\otimes_A\bar{A})^{[1]}\not\cong \bar{A}^{[1]}\otimes_{A^{[1]}}\bar{A}^{[1]}=\Bbbk[y^2]\otimes_{\Bbbk[y^4]}\Bbbk[y^2]
		\]
	\end{remark}
	
	\begin{remark}
		We caution that $\ol{A^{[1]}}$ is not isomorphic to $\bar{A}^{[1]}$, as the example in Remark \ref{rmk caution twist surj} shows.
	\end{remark}
	
	\begin{definition}
		For a ringed space $\XX=(|\XX|,\OO_{\XX})$, write $\OO_X^{[1]}$ for the sheafification of the presheaf $U\mapsto\OO_X(U)^{[1]}$.  Then define $\XX^{[1]}$ to be the ringed space $(|\XX|,\OO_{X}^{[1]})$, and write \linebreak $\phi_{\XX}:\XX\to\XX^{[1]}$ for the natural map of locally ringed spaces.
	\end{definition}

	\begin{lemma}\label{lemma fr affine scheme}
		If $\XX=\operatorname{Spec}A$, then $\XX^{[1]}=\operatorname{Spec}A^{[1]}$.  
	\end{lemma}
	\begin{proof}
		This follows from (4) of Lemma \ref{lemma Fr and ring maps localization}, and the fact that the distinguished open subschemes $\DD(f^p)\sub\XX$ give a basis for the topology, where $f$ runs over elements of $A_{(0)}$.
	\end{proof}
	
	\begin{lemma}\label{lemma finite univ homeo}
		If $\XX$ is an algebraic scheme, then so is $\XX^{[1]}$, and we have $\XX(\Bbbk)=\XX^{[1]}(\Bbbk)$.  Further, $\phi_{\XX}$ is a finite, universal homeomorphism.
	\end{lemma}
	
	\begin{proof}
		The ringed space $\XX^{[1]}$ is again an algebraic scheme by Lemma \ref{lemma fr affine scheme} and (3) of Lemma \ref{lemma A finite over Fr(A)}.  The fact that $\phi_{\XX}$ is finite follows from (2) of Lemma \ref{lemma A finite over Fr(A)}.  Finally, $\phi_{\XX}$ is a universal homeomorphism by (5) of Lemma \ref{lemma A finite over Fr(A)} and the fact that this may be checked affine locally.
	\end{proof}
	
	Given a morphism $f:\XX\to\YY$ of schemes, write $f^{[1]}:\XX^{[1]}\to\YY^{[1]}$ for the induced map on the Frobenius twists.  It is not difficult to see that $f^{[1]}$ is again a morphism of locally ringed spaces, and thus is a morphism of schemes.  
	
	\begin{lemma}\label{lemma fr twist morphisms schemes}  Let $\XX,\YY$ be algebraic schemes.
		\begin{enumerate}
			\item We have a canonical isomorphism $(\XX\times\YY)^{[1]}\cong\XX^{[1]}\times\YY^{[1]}$.
			\item If $f:\XX\to\YY$ is finite, then so is $f^{[1]}:\XX^{[1]}\to\YY^{[1]}$.
			\item If $f:\XX\to\YY$ is an open immersion, then so is $f^{[1]}:\XX^{[1]}\to\YY^{[1]}$.
			\item  $\XX$ is separated, affine, quasi-affine, or proper if and only if $\XX^{[1]}$ is likewise. 
		\end{enumerate}
	\end{lemma}
	\begin{proof}
		Part (1) follows from (1) of Lemma \ref{lemma Fr and ring maps localization}, and (2) follows from (3) of the same lemma.  Part (3)  may be checked affine locally on $\YY$, so suppose that $\YY=\Spec A$.  Then we may write $\XX=\cup_i D(f_i^p)$ for some $f_i\in A_{(0)}$.  Now Lemma \ref{lemma Fr and ring maps localization} implies that $\XX^{[1]}=\cup_iD(f_i^p)^{[1]}$, and thus remains an open subscheme of $\YY^{[1]}$.  
		
		Finally, we prove (4).  Since $\phi_{\XX}:\XX\to\XX^{[1]}$ is a homeomorphism, their diagonal morphisms will give the same topological subspace of $|\XX\times\XX|=|\XX^{[1]}\times\XX^{[1]}|$.  Therefore one image is closed if and only if the other is.  This proves $\XX$ is separated if and only if $\XX^{[1]}$ is.  On the other hand, $\phi_{\XX}:\XX\to\XX^{[1]}$ is affine, so we obtain that $\XX$ is affine whenever $\XX^{[1]}$ is and the converse follows from Lemma~\ref{lemma fr affine scheme}.  For quasi-affinity, we use Part (5) of the characterization given in \cite[Prop.~5.4.2]{C3}, and that $(\XX_{f^p})^{[1]}\cong(\XX^{[1]})_{f^p}$ for all $f\in\Gamma(\XX,\OO_{\XX})_{(0)}$ by (4) of Lemma \ref{lemma Fr and ring maps localization}.  Finally for properness, let $\YY$ be any scheme and consider the commutative diagram:
		\[
		\xymatrix{
			\XX\times\YY\ar[rr]\ar[rd] & & \XX^{[1]}\times\YY.\ar[dl]\\
			& \YY &
		}
		\]
		By Lemma \ref{lemma finite univ homeo}, the horizontal arrow is a homeomorphism.  Therefore one diagonal arrow is a closed morphism if and only if the other is.  It follows that $\XX\to\Spec\Bbbk$ is universally closed if and only if $\XX^{[1]}\to\Spec\Bbbk$ is universally closed.  Since $\XX$ is separated if and only if $\XX^{[1]}$ is, this shows the same is true for properness.
	\end{proof}
	
	\begin{remark}
		We again caution that while we have $(\XX\times\YY)^{[1]}\cong\XX^{[1]}\times\YY^{[1]}$, it is not true that Frobenius twist will commute with arbitrary fibre products.  Further, while Frobenius twists preserve open embeddings, they need not preserve closed embeddings (see Remark \ref{rmk caution twist surj}).
	\end{remark}
	
	We conclude this subsection with a powerful lemma which we will not use in this paper. 
	\begin{lemma}\label{LemFinSim}
		If $\mathscr{C}_{fin}$ is a union of tensor subcategories with finitely many simples, then for an algebraic scheme $\XX$ in $\mathscr{C}$ there exists $r>0$ such that $\XX^{[r]}$ is purely even.
	\end{lemma}

	\begin{proof}
		Since $\XX$ is algebraic, it must actually be a scheme in (the ind-completion of) a tensor subcategory in the union. Indeed, a finitely generated algebra $A$ in $\mathscr{C}$ is generated by a compact object, which thus lies in a tensor category in the union, and an algebraic scheme has a finite cover by spectra of such finitely generated algebras. Hence, we can assume that $\mathscr{C}$ itself has finitely many simple objects.
		
		 Then, by \cite[Cor.~10.3]{EO}, applying enough Frobenius twists will take $\mathscr{C}$ into a Frobenius exact tensor subcategory of $\Cs_{ex}$ of $\Cs$.  Since $\Cs$ satisfies (MN1-2), $\Cs_{ex}$ will also, which implies $\Cs_{ex}$ is a tensor subcategory of $\operatorname{Ver}_p$ by \cite[Theorem~7.2.3(2)]{CEO1} and \cite[Theorem~1.1]{CEO2}.   It thus suffices to assume $\Cs=\operatorname{Ver}_p$.  By Example \ref{example verp frob twist}, we see that $(\Spec A)^{[1]}$ will be purely even for any commutative algebra $A$, and therefore $\XX^{[1]}$ is purely even for any scheme $\XX$.
	\end{proof}
	
	\begin{remark}
		The main example of our theory is intended for is $\cC=\mathrm{Ver}_{p^\infty}$, for which the assumption in Lemma~\ref{LemFinSim} is satisfied.
	\end{remark}
	
	\subsection{Frobenius twists of algebraic groups}
	\begin{lemma}\label{lemma twist hopf alg}
		Suppose that $A$ is a commutative Hopf algebra.
		\begin{enumerate}
			\item $A^{[1]}$ is a Hopf subalgebra of $A$.
			\item Write $\m$ for the augmentation ideal of $A$, and $\m^{[1]}$ for the image of $\operatorname{Fr}(\m)\to\m$.  Then $\m^{[1]}$ is the augmentation ideal of $A^{[1]}$, and $\m^{[1]}\sub\m^p$.
		\end{enumerate}
	\end{lemma}
	
	\begin{proof}
		For (1), we use again that $\operatorname{Fr}$ is symmetric monoidal to obtain that $\operatorname{Fr}(A)$ is a Hopf algebra and $\phi_A:\operatorname{Fr}(A)\to A$ is a morphism of Hopf algebras.
		
		For (2), we have a splitting $A=\m\oplus\mathbf{1}$.  By the additivity of $\operatorname{Fr}$, we have $\operatorname{Fr}(A)=\operatorname{Fr}(\mathbf{1})\oplus\operatorname{Fr}(\operatorname{\m})$, so that $A^{[1]}=\mathbf{1}\oplus\m^{[1]}$, implying that $\m^{[1]}$ is the augmentation ideal of $A^{[1]}$.  Further, $\m^{[1]}$ is the image of $\operatorname{Fr}(\m)\to\m$ which is induced by the multiplication morphism
		\[
		\m^{\otimes p}\to \m,
		\]
		and it is clear that the image lands in $\m^p$.  From this the statement follows.
	\end{proof}

	Let $\GG$ be an  algebraic group, so that $\GG^{[r]}:=\Spec(\Bbbk[\GG]^{[r]})$ will again be an  algebraic group for $r\in\mathbb{N}$.  We have a finite quotient of group schemes $\GG\to\GG^{[r]}$, and we define $\GG_r$ to be the kernel of this quotient.  We in turn obtain a Hopf subalgebra $\Gamma_{\GG_r}\sub\Gamma_\GG$, and by Lemma \ref{lemma ses dist algs} a short exact sequence:
	\[
	0\to(\Gamma_{\GG_r}^+)\to\Gamma_\GG\to\Gamma_{\GG^{[r]}}\to 0.
	\]

	\begin{lemma}\label{lemma gamma Gr has low order dists}
		For $r\in\N$ and for $n<p^r$, the inclusion $\GG_r\leq\GG$ induces an isomorphism 
		\[
		\Gamma_{\GG_{r}}^{n}\xto{\sim}\Gamma_{\GG}^{n}.
		\]
		
	\end{lemma}
    
	\begin{proof}
	   By repeated application of Lemma \ref{lemma twist hopf alg}, we obtain that $\m^{[r]}\sub\m^{p^r}$. Hence for $n<p^r$, we obtain isomorphisms
       \[(\Gamma_\GG^n)^\ast=\Bbbk[\GG]/\mathfrak{m}^{n+1}\to \Bbbk[\GG]/(\mathfrak{m}^{n+1}+\mathfrak{m}^{[r]})=(\Gamma_{\GG_r}^n)^\ast,\]
       from which the result follows.
	\end{proof}
	
	\begin{prop}\label{prop twist alg gp}
		Suppose that $\GG$ is an algebraic group and $r\in\mathbb{N}$. \begin{enumerate}
			\item $\GG_r$ is an infinitesimal group scheme with $\operatorname{Lie}\GG_r=\g$.
			\item For $r$ large enough, $\GG/_{\hspace{-0.2em}gp}\GG_r\cong\GG^{[r]}$ is a purely even algebraic group.
		\end{enumerate}  
	\end{prop}
	\begin{proof}
		Since the morphism $\GG\to\GG^{[r]}$ is finite and gives an isomorphism on $\Bbbk$-points by Lemma~\ref{lemma finite univ homeo}, we obtain that $\GG_r$ is an infinitesimal group scheme, proving (1).  For (2), we apply Corollary \ref{cor fg gamma} and Lemma \ref{lemma gamma Gr has low order dists} which imply that $\Gamma_{\GG}$ is generated as an algebra by $\Gamma_{\GG_0}$ and $\Gamma_{\GG_r}^+$ for some $r>0$.  Thus $\Gamma_{\GG^{[r]}}\cong\Gamma_{\GG}/(\Gamma_{\GG_r}^+)$ will be purely even, which implies that $\GG^{[r]}$ is also purely even by Lemma \ref{lemma gamma even GG even}.
	\end{proof}
	
	\begin{cor}
		For all $n\in\N$, $\Gamma_{\GG}^n$  lies in a finite-length Hopf subalgebra of $\Gamma_\GG$. 
	\end{cor}
	\begin{proof}
		Indeed, $\Gamma^{n}_{\GG}\sub\Gamma_{\GG_r}$ when $n<p^r$.
	\end{proof}
	
	\begin{remark}
		If $\HH\leq\GG$ is a closed subgroup, then we have a natural map $\HH^{[1]}\to\GG^{[1]}$, but this need not be an immersion.  Indeed, a counterexample is given by $\G_a\leq\PP$ in Remark \ref{rmk caution twist surj}.
	\end{remark}
	
	\subsection{Properties determined by $\Bbbk$-points}
	
	%\begin{lemma}\label{lemma twist enough smooth}
	%    If $\GG$ is an algebraic group, then $\GG^{[r]}$ is a smooth, purely even algebraic group for $r$ large enough.
	%\end{lemma}
	%\begin{proof}
	%    By Proposition \ref{prop twist alg gp}, $\GG^{[r]}$ is a purely even algebraic group for $r$ large enough.  Increasing $r$, we may remove any nilpotents from the coordinate ring, and thus obtain that $\GG^{[r]}$ is smooth as desired.
	%\end{proof}

	\begin{lemma}\label{LemHH}
		Let $\HH'\leq\HH$ be algebraic groups.  If $\HH'(\Bbbk)=\HH(\Bbbk)$, then $\Gamma_{\HH}$ is a finitely generated $\Gamma_{\HH'}$-module.
	\end{lemma}
	\begin{proof}
		By \cite[Thm.~III]{Sch}, we have an isomorphism of right $\Gamma_{\HH'_0}$-modules $\Gamma_{\HH_0}=D\otimes\Gamma_{\HH'_0}$, where $D:=\Gamma_{\HH_0}/\Gamma_{\HH_0}\Gamma_{\HH'_0}^+=\operatorname{Dist}(\HH_0/\HH'_0,e\HH'_0)$. By assumption $(\HH_0/\HH'_0)(\Bbbk)=\HH_0(\Bbbk)/\HH'_0(\Bbbk)$ is a singleton and thus $\HH/\HH'$ is infinitesimal. Hence $D$ is finite and  $\Gamma_{\HH_0}$ is finitely generated over $\Gamma_{\HH_0'}$.  This in turn implies that $\Gamma_{\HH}$ is finitely generated over $\Gamma_{\HH'}$ by Corollary~\ref{cor fg gamma}.
	\end{proof}
	
	Given subgroups $\NN,\HH\leq\GG$ such that $\NN$ is normal in $\GG$, then we write $\NN\HH$ for the subgroup of $\GG$ given by the image of the natural map $\NN\rtimes\HH\to\GG$ (see Section \ref{sec subgrps}).  By Lemma \ref{lemma semidirect prod dist alg}, we have that $\Gamma_{\NN\HH}=\Gamma_{\NN}\Gamma_{\HH}\sub\Gamma_{\GG}$.

	\begin{cor}\label{cor k pts subgp}
		Let $\HH'\leq\HH$ be algebraic groups with $\HH'(\Bbbk)=\HH(\Bbbk)$. There exists $r>0$ such that $\Gamma_{\HH}=\Gamma_{\HH'}\Gamma_{\HH_r}$. Further, for such an $r$ we have that $\HH'\to\HH^{[r]}$ is a quotient.
	\end{cor}
	\begin{proof}
		The first statement follows from combining Lemma~\ref{LemHH} with Lemma~\ref{lemma gamma Gr has low order dists}.  For the second statement, for such $r$ we have that $\HH'\HH_r=\HH$ by Corollary \ref{cor iso alg gps criteria}, and thus $\HH'\HH_r\to\HH^{[r]}$ and then also $\HH'\to\HH^{[r]}$ is a quotient.
	\end{proof}

   For the following, recall the subgroup $\GG_{\mathrm{red}}\leq\GG$ from Remark \ref{rem:red}.
	\begin{cor}\label{cor G0 to Gr surj}
		There exists $r>0$ such that both $\GG_{\mathrm{red}}\to\GG^{[r]}$ and $\GG_{0}\to\GG^{[r]}$ are quotients; in particular $\GG=\GG_r\GG_{\mathrm{red}}=\GG_{r}\GG_0$.  
	\end{cor}
	\begin{proof}
		Indeed, $\GG_{\mathrm{red}}(\Bbbk)=\GG_0(\Bbbk)=\GG(\Bbbk)$, so we may apply Corollary \ref{cor k pts subgp}.
	\end{proof}
	
	\section{Generalities on quotients}\label{Faisc}
	
	Let $\XX$ be a scheme and $\HH$ an algebraic group in $\mathscr{C}$, where we recall that (GR) and (MN1-2) hold.  A right action of $\HH$ on $\XX$ is a morphism $a:\XX\times\HH\to \XX$ of schemes satisfying the usual commutative diagrams.
	
	Given a right $\HH$-action on $\XX$ we define the quotient scheme, if it exists, to be a scheme $\YY$ with a morphism $\pi:\XX\to \YY$ such that the following is a co-equalizer in the category of all schemes:
	\[
	\left(\XX\times\HH\rightrightarrows \XX\right)\xto{\pi}  \YY
	\]
	If the quotient $(\YY,\pi)$ exists, then we will write $\YY=\XX/\HH$.  Notice that if the quotient exists, then $\pi$ is necessarily an epimorphism.
	
	We say that an action $a:\XX\times\HH\to\XX$ is \emph{free} if $pr_{1}\times a:\XX\times\HH\to\XX\times\XX$ is a closed embedding.  The following is proven in exactly the same way as Chpt.~III, $\S$1, 2.4 and $\S3$, 2.5, and Chpt.~I, $\S2$, 3.9 of \cite{DG}.
	
	\begin{lemma}\label{lemma quot fppf can}
		If $\HH$ acts freely on $\XX$ and the quotient $(\YY,\pi)$ exists, then $\XX\times\HH\cong \XX\times_{\YY}\XX$ and $\pi$ is faithfully flat and affine. 
	\end{lemma}

	\subsection{The fppf topology}\label{sec fppf topology} It is well known that quotient schemes need not exist already when $\mathscr{C}=\Vec$.  However if we work in the category of sheaves in the fppf topology, which we call faisceaux, then a quotient always exists.  Before we explain this, let us recall a few details about the fppf topology.  A morphism $R\to R'$ is called fppf if it is faithfully flat and finitely presented.  Open covers are collections of morphisms $R\to R_i$ such that $R\to \prod R_i$ is fppf.  One can show that a presheaf $\FF$ is a faisceau if and only if it preserves products and for every fppf morphism $R\to R'$, the following is an equalizer:
	\[
	\FF(R)\to\left(\FF(R')\rightrightarrows \FF(R'\otimes_RR')\right).
	\]
	The following is standard, see \cite[Lemma~4.3.5]{C3}.
	\begin{lemma}\label{lemma schemes are sheaves}
		If $\XX$ is a scheme, then $\XX$ is a sheaf in the fppf topology.
	\end{lemma}
	\subsection{Affine quotients} Suppose that $\XX=\Spec A$ is affine with a free action of $\HH$, and let $B=A^{\HH}$, so that we have a natural map $\XX\to \YY=\Spec B$.
	
	\begin{lemma}\label{lemma quots affine}
		We have $\YY=\XX/\HH$ if and only if $B\to A$ is fppf and the `canonical' map (following the terminology of \cite{Sch})
		\[
		\mathbf{can}:A\otimes_{B}A\to A\otimes\Bbbk[\HH],
		\]
		given by $i_1\otimes a$ is an isomorphism.  Here $i_1$ is inclusion of $A$ into the first tensor, and $a$ is the coaction $A\to A\otimes\Bbbk[\HH]$.
	\end{lemma}
	\begin{proof}
		The forward direction is Lemma \ref{lemma quot fppf can}.  For the converse, our second condition implies that we have an isomorphism of equalizer diagrams 
		\[
		(A\rightrightarrows A\otimes\Bbbk[\HH])\cong (A\rightrightarrows A\otimes_{B}A),
		\]
        and thus an isomorphism of co-equalizer diagrams:
        \[
		(\XX\times\HH\rightrightarrows \XX)\cong (\XX\times_{\YY}\XX\rightrightarrows \XX).
		\]
        The right diagram has co-equaliser $\YY$, by \cite[Proposition~4.26]{C3}, and thus the left diagram does too. 
	\end{proof}
	
	\subsection{The faisceau $\XX/_{\hspace{-0.2em}f}\HH$}  Suppose that $\HH$ acts freely on $\XX$; the case of particular interest to us is when $\HH$ is a closed subgroup of $\GG$, and $\HH$ acts on $\GG$ by right multiplication. Write $\XX/_{\hspace{-0.2em}0}\HH$ for the presheaf given by
	\[
	A\mapsto \XX(A)/\HH(A).
	\]
	This is clearly the quotient in the category of presheaves.  We set $\XX/_{\hspace{-0.2em}f}\HH$ to be the associated faisceau of $\XX/_{\hspace{-0.2em}0}\HH$.  By its universal property, $\XX/_{\hspace{-0.2em}f}\HH$ will be the quotient of $\XX$ by $\HH$ in the category of faisceaux.  It has the following explicit description: for an algebra $A$, we have
	\[
	(\XX/_{\hspace{-0.2em}f}\HH)(A)=\lim\limits_{\rightarrow}\operatorname{eq}[\XX(B)/\HH(B)\rightrightarrows \XX(B\otimes_AB)/\HH(B\otimes_AB)],
	\]
	where $B$ runs over all fppf maps $A\to B$ (see, for instance, \cite[Sec.~5.4]{J}). 
	
	\begin{lemma}\label{lemma quot exists in sch iff fais}
		If $\XX/\HH$ exists, then it represents $\XX/_{\hspace{-0.2em}f}\HH$.
	\end{lemma}
	\begin{proof}
		By Lemma~\ref{lemma quot fppf can}, we have an isomorphism of diagrams 
		\[
		(\XX\times\HH\rightrightarrows\XX)\cong(\XX\times_{\XX/\HH}\XX\rightrightarrows\XX)
		\]
		Clearly, $\XX/_{\hspace{-0.2em}f}\HH$ is the co-equalizer of the LHS diagram in the category of faisceaux.  On the other hand, because $\XX\to\XX/\HH$ is faithfully flat and affine (Lemma~\ref{lemma quot fppf can}), $\XX/\HH$ is the co-equalizer of the RHS diagram in the category of faisceaux.   Indeed we may check this affine locally on $\XX/\HH$, and there it follows from the definition of a faisceau.  Thus $\XX/\HH$ and $\XX/_{\hspace{-0.2em}f}\HH$ are co-equalizers of the same diagram in the category of faisceaux.
	\end{proof}
	
	\begin{cor}\label{cor k pts quotient}
		If $\GG/\HH$ exists, then $(\GG/\HH)(\Bbbk)=\GG(\Bbbk)/\HH(\Bbbk)$.
	\end{cor}
	\begin{proof}
		Indeed by \cite[Chpt.~III, \S 1, 1.15]{DG}, we have $(\GG/_{\hspace{-0.2em}f}\HH)(\Bbbk)=\GG(\Bbbk)/\HH(\Bbbk)$, so we may apply Lemma \ref{lemma quot exists in sch iff fais}.
	\end{proof}

	\begin{lemma}\label{lemma quot in Vec also in C}
		Suppose that $\HH_0\leq\GG_0$ are algebraic groups in $\operatorname{Vec}_{\Bbbk}$.  Then $\GG_0/\HH_0$ exists in the category of schemes over $\mathscr{C}$.  In particular, $\GG_0/\HH_0$ represents $\GG_0/_{\hspace{-0.2em}f}\HH_0$.
	\end{lemma}
	
	\begin{proof}
		Indeed, the inclusion functor $\operatorname{Sch}_{\Bbbk}\to \operatorname{Sch}_{\mathscr{C}}$ is left adjoint to the functor $\XX\mapsto\XX_0$, and thus preserves colimits.    For the second statement, we may apply Lemma \ref{lemma quot exists in sch iff fais}. 
	\end{proof}

	\section{Quotients of additive group schemes}\label{Additive}
	
	As a warm-up case, we discuss the existence and explicit description of quotients of additive group schemes. These results will not be used for the general case in the next section, and also do not rely on the theory developed in Section~\ref{FrobTw}.

    Let $0\to X\to Y\to Z\to0$ be a short exact sequence of compact objects in $\mathscr{C}$.  Write $\XX,\YY,\ZZ$ for the corresponding additive group schemes in $\mathscr{C}$, where e.g. $\XX(A)=(X\otimes A)_{(0)}$ for a commutative algebra $A$ in $\mathscr{C}$.  Alternatively, $\XX=\Spec S(X^*)$, so $\XX,\YY$ and $\ZZ$ are algebraic groups in $\mathscr{C}$.   In this case, $\Gamma_{\XX}$ is the divided power algebra of $X$, {\it i.e.} $\oplus_n (X^{\otimes n})^{S_n}$. 
	
	We have an induced algebra morphism $S(Z^*)\to S(Y^*)$ whose image\footnote{If $\mathscr{C}=\mathrm{Ver}_p$, since $\cC$ is semisimple, we simply have $R= S(Z^\ast)$. However, in $\mathrm{Ver}_{p^n}$ for $n>1$ there are short exact sequences for which $R$ is a proper quotient of $S(Z^\ast)$.} we call $R$, and a surjective map $S(Y^*)\to S(X^*)$.  
	We see that $Y^*$ admits a natural filtration $Z^*\sub Y^*$, and this induces a corresponding $\N$-filtration on $S(Y^*)$.  Then \cite[Thm.~3. 3.2]{CEO1} states that with respect to this filtration we have
	\[
	\operatorname{gr} S(Y^*)\cong R\otimes S(X^*),
	\]
	where the grading on the RHS comes from the standard one on $S(X^*)$.  It follows that $S(Y^*)$ admits a filtration by $R$-modules whose layers are faithfully flat, and therefore $S(Y^*)$ is a faithfully flat $R$-algebra, and thus fppf.  Alternatively, this follows from \cite[Theorem~6.3.1]{C2}.

    \subsection{The quotient} We have seen that $\XX$ is a subgroup of $\YY$, and now we describe the scheme quotient $\YY/\XX$.
	Consider the canonical map of algebras, from Lemma \ref{lemma quots affine}:
	\[
	\mathbf{can}:S(Y^*)\otimes_{R} S(Y^*)\to S(Y^*)\otimes S(X^*).
	\]
	View $S(X^*)$ as a graded algebra with $X^*$ in degree 1, and give it the natural induced filtration.  We continue to view $S(Y^*)$ as a filtered algebra arising from the inclusion $Z^*\sub Y^*$.  Give $S(X^*)\otimes S(Y^*)$ and $S(Y^*)\otimes_{R} S(Y^*)$ the total filtrations coming from the natural bifiltrations on each.  Then this morphism respects the filtration, and taking associated graded we obtain:
	\[
	\operatorname{gr}\mathbf{can}:(R\otimes S(X^*))\otimes_{R}(R\otimes S(X^*))\to (R\otimes S(X^*))\otimes S(X^*).
	\]
	This map is exactly the base change to $R$ of the canonical map for the action of $\mathcal{X}$ on itself, i.e.
	\[
	\mathbf{can}_{X}:S(X^*)\otimes S(X^*)\to S(X^*)\otimes S(X^*),
	\]  
	and this map is an isomorphism.  It follows that our original map was also an isomorphism, and thus we obtain by Lemma \ref{lemma quots affine}:
	\[
	\YY/\XX\cong \Spec R.
	\]
	
	\begin{example}\label{example ver4+ homog space}
		Let $\mathscr{C}=\operatorname{Ver}_4^+$, and let $\PP$ denote the additive group scheme on the projective indecomposable object $P$.  Then we have a subgroup $\G_a\leq\PP$, and by Example \ref{example ver4+}, we have $\PP/\G_a\cong\Spec\Bbbk[u]/u^2$, and this is a purely even affine scheme.   In particular $(\PP/\G_a)_0\cong \Spec\Bbbk[u]/u^2$, while $\PP_0/(\G_a)_0\cong \G_a/\G_a\cong \Spec\Bbbk$.
        
		Distinguished low-dimensional algebraic groups, such as $\PP$, in $\operatorname{Ver}_4^+$ were studied more in depth in \cite{H}.
		
	\end{example}
	
	\begin{remark}
		Example \ref{example ver4+ homog space} in particular illustrates that if $0\to X\to Y\to Z\to 0$ is a short exact sequence of compact objects in $\mathscr{C}$, we obtain a left exact sequence of additive algebraic groups $0\to\XX\to\YY\to\ZZ$, and this need not be exact.  However this sequence will be exact if $\mathscr{C}$ is Frobenius exact, i.e.~$\mathscr{C}\sub\operatorname{Ver}_p$.  
	\end{remark}
	\section{Existence and properties of homogeneous spaces}\label{MainSec}
	
	In this section, we closely follow the idea suggested in \cite[Rmk.~9.11]{MZ}.  
	
	\subsection{Existence}
	We first need some preparatory lemmas.  Recall the definition of an infinitesimal scheme in Definition \ref{defn infinitsml}.

	\begin{lemma}\label{lemma infinitesimal quot exist}
		Suppose that an algebraic group $\HH$ acts on an infinitesimal scheme $\XX$.  Then $\XX/\HH$ exists and is equal to $\operatorname{Spec}\Bbbk[\XX]^{\HH}$.
	\end{lemma}
	\begin{proof}
		First, an observation: let $\ZZ$ be any infinitesimal scheme and let $\YY$ be any scheme. Then since $|\ZZ|$ is a singleton, any morphism $\ZZ\to \YY$ must factor through an open affine subscheme of $\YY$.
		
		Now we prove the lemma. By construction, $\Spec\Bbbk[\XX]^{\HH}$ satisfies the universal property of the quotient $\XX/\HH$ inside the category of affine schemes, so let us show that it also satisfies the universal property in the category of all schemes.  Let $\YY$ be a scheme admitting a map $p:\XX\to\YY$ from the equalizer $(\XX\times\HH\rightrightarrows\XX)$.  Since $\XX$ is infinitesimal, $p$ factors through an affine subscheme $\Spec B$ of $\YY$.  Hence $p$ factors through $\XX\to\Spec\Bbbk[\XX]^{\HH}$.  
	\end{proof}
	
	If $\XX,\YY$ are faisceaux then we say a morphism $f:\XX\to\YY$ is \emph{affine} if for any morphism $\Spec R\to\YY$ we have that $\Spec R\times_{\YY}\XX$ is an affine scheme.  The following is \cite[Chpt.~I, \S 2, Prop.~3.3]{DG}.  
	\begin{lemma}\label{lemma affine fais scheme}
		Let $\XX,\YY$ be faisceaux.  If $f:\XX\to\YY$ is an affine morphism and $\YY$ is a scheme, then $\XX$ is also a scheme.
	\end{lemma}
	
	The following is \cite[Chpt.~III, \S 1, Cor.~2.12]{DG}.
	\begin{lemma}\label{LemCart}
		Suppose we have a Cartesian square of faisceaux:
		\[
		\xymatrix{
			\XX \ar[r]\ar[d]_{u} & \YY \ar[d]^{v}\\
			\mathcal{W} \ar[r]_{p} & \ZZ
		}
		\]
		If $u$ is affine and $p$ is an epimorphism, then $v$ is also affine.
	\end{lemma}
	
	Let $\GG$ be an algebraic group, and use Lemma \ref{cor G0 to Gr surj} to find $r\in\N$ such that $\GG_0\to \GG^{[r]}\cong\GG/_{\hspace{-0.2em}gp}\GG_r$ is a quotient.  Let $\HH\leq\GG$ be a subgroup, and write $\widetilde{\HH}:=\GG_r\HH$.  Since $\GG_r$ is a normal subgroup, $\widetilde{\HH}$ is a subgroup of $\GG$ with $\HH\leq\widetilde{\HH}\leq \GG$.  
	
	\begin{lemma}
		The quotient scheme $\widetilde{\HH}/\HH$ exists and is affine.
	\end{lemma}
	\begin{proof}
		We have a natural isomorphism of faisceaux $\GG_r/_{\hspace{-0.2em}f}(\GG_r\cap\HH)\cong \widetilde{\HH}/_{\hspace{-0.2em}f}\HH$.  Since $\GG_r$ is infinitesimal by Proposition \ref{prop twist alg gp}, we may now apply Lemmas \ref{lemma infinitesimal quot exist} and \ref{lemma quot exists in sch iff fais}. 
	\end{proof}
	
	\subsection{Main theorem and consequences}
	
	\begin{thm}\label{Main}
		Let $\GG$ be an algebraic group in $\mathscr{C}$, and let $\HH\leq\GG$ be a subgroup.  Then 
		\begin{enumerate}
			\item $\GG/\HH$ exists as a scheme, is algebraic and separated.
			\item The natural map $\GG_0/\HH_0\to\GG/\HH$ is a closed immersion and a universal homeomorphism.  In particular, the same is true for $\GG_0/\HH_0\to(\GG/\HH)_0$.
			\item $\GG/\HH$ is affine, quasi-affine, or proper if and only if $\GG_0/\HH_0$ is likewise.
		\end{enumerate}  
	\end{thm}
	
	We prove Theorem \ref{Main} in a couple of steps.
	
	\begin{proof}[Proof of existence of $\GG/\HH$] We continue working with the algebraic group $\widetilde{\HH}$ as introduced in the previous subsection.  We have the following Cartesian square:
		\[
		\xymatrix{
			\GG\times \widetilde{\HH}/_{\hspace{-0.2em}f}\HH\ar[r]^f\ar[d]_u & \GG/_{\hspace{-0.2em}f}\HH\ar[d]^v\\
			\GG\ar[r]_p & \GG/_{\hspace{-0.2em}f}\widetilde{\HH}
		}
		\]
		Indeed, we have such a Cartesian square in the category of presheaves if we take the set-theoretic quotients, and then we may use that sheafification commutes with finite limits.
		The morphism $u$ is the projection onto $\GG$, $p$ is the canonical quotient morphism, $v$ is the morphism induced by the universal property of quotients, and $f$ is given by $(g,h'\HH(A))\mapsto(gh'\HH(A))$ for an algebra $A$.  The quotient morphism $p$ is always an epimorphism and the morphism $u$ is affine since $\widetilde{\HH}/\HH=\widetilde{\HH}/_{\hspace{-0.2em}f}\HH$ is affine by Lemma \ref{lemma infinitesimal quot exist}. In particular, $v$ is affine by Lemma~\ref{LemCart}.

		Observe that we have an isomorphism of faisceaux
		\[
		\GG_0/_{\hspace{-0.2em}f}(\GG_0\cap\widetilde{\HH})\xto{\sim} \GG/_{\hspace{-0.2em}f}\widetilde{\HH},
		\]
		which arises from the universal property of $\GG_0/_{\hspace{-0.2em}f}(\GG_0\cap\widetilde{\HH})$.  It is clearly a monomorphism of faisceaux; to see that it is an {epimorphism}, we use that $\GG_0\to\GG^{[r]}=\GG/_{\hspace{-0.2em}gp}\GG_r$ is a quotient of faisceaux by \cite[Prop.~7.2.4 (1)]{C2}.   Now we may use that a morphism of faisceaux is an isomorphism if and only if it is a monomorphism and an epimorphism. 
		
		The faisceau $\GG_0/_{\hspace{-0.2em}f}(\GG_0\cap\widetilde{\HH})$ is a quotient of even group schemes, and thus exists by Lemma~\ref{lemma quot in Vec also in C}. Thus $\GG/\widetilde{\HH}$ exists, and so by Lemma~\ref{lemma affine fais scheme} it follows that $\GG/\HH$ exists as a scheme as well.
	\end{proof}
	
	Suppose that $\GG$ acts on a scheme $\XX$.  Then it is easy to check that $\GG^{[1]}$ will naturally act on $\XX^{[1]}$.  In particular, $\GG^{[1]}$ acts on $(\GG/\HH)^{[1]}$.
	
	\begin{lemma}\label{Lem Twist Hom}
		For $r\in\N$, $(\GG/\HH)^{[r]}$ is a homogeneous space for $\GG^{[r]}$.  
	\end{lemma}
	
	\begin{proof}
		It suffices to check this for $r=1$.  Consider the orbit map $q:\GG^{[1]}\to(\GG/\HH)^{[1]}$ at the base point $e=\HH(\Bbbk)$. Writing $\HH'\leq\GG^{[1]}$ for the isotropy subgroup of $e$, we have an immersion $p:\GG^{[1]}/\HH'\to(\GG/\HH)^{[1]}$ by \cite[Lem.~7.25]{C3}.  Since $\GG^{[1]}(\Bbbk)=\GG(\Bbbk)$ acts transitively on $(\GG/\HH)^{[1]}(\Bbbk)=(\GG/\HH)(\Bbbk)$ by Corollary \ref{cor k pts quotient}, the morphism $p$ must be a closed immersion.

		To show that $p$ is an isomorphism, consider the following diagram of sheaves of commutative rings on $|\GG/\HH|$:
		\[
		\xymatrix{
			\OO_{\GG/\HH} \ar@{^{(}->}[rr] & & q_*\OO_{\GG} \\
			\OO_{\GG/\HH}^{[1]}\ar@{^{(}->}[u]\ar[r]^{p^*} & p_*\OO_{\GG^{[1]}/\HH'} \ar@{^{(}->}[r] & q_*\OO_{\GG}^{[1]}\ar@{^{(}->}[u]
		}
		\]
		We want to show that $p^*$ is an isomorphism, and we already know it is an epimorphism since $p$ is a closed embedding.  However, all other morphisms in the diagram are monomorphisms, thus $p^*$ must be also.  This completes the proof.
	\end{proof}
	
	\begin{remark}
		There is a natural map $\phi:\HH^{[1]}\to \HH'$, where $\HH'\leq\GG^{[1]}$ is the subgroup appearing in the proof of Lemma \ref{Lem Twist Hom}.  However this map need not be an isomorphism.  An example is provided by Example \ref{example ver4+ homog space}, where $\HH^{[1]}\cong\HH'\cong\mathbb{G}_a$, and the morphism $\phi$ is the quotient morphism arising from the Frobenius morphism.  We do not know of any examples where $\phi$ is not a quotient morphism.
	\end{remark}

	\begin{proof}[Finishing the proof of Theorem \ref{Main}.]
		Let $r\in\N$ be such that $\GG^{[r]}$ is purely even.  Then $(\GG/\HH)^{[r]}$ is a homogeneous space for $\GG^{[r]}$ by Lemma~\ref{Lem Twist Hom}, and thus is a separated algebraic scheme by \cite[Chpt.~III, \S3, 5.4]{DG}. Since the morphism $\GG/\HH\to(\GG/\HH)^{[r]}$ is finite by Lemma~\ref{lemma finite univ homeo}, $\GG/\HH$ must also be a separated, algebraic scheme.  Further, by Lemma \ref{lemma fr twist morphisms schemes}, $\GG/\HH$ is affine, quasi-affine, or proper if and only if $(\GG/\HH)^{[r]}$ is likewise.
		
		By Corollary \ref{cor G0 to Gr surj}, $\GG_0\to\GG^{[r]}$ is a quotient for some $r\in\N$.  Thus $(\GG/\HH)^{[r]}$ is a homogeneous space for $\GG_0$, which we may write as $\GG_0/\HH'$.  Clearly, the orbit map at the base point factors through $\GG_0/\HH_0\to\GG_0/\HH'$, implying that $\HH_0\leq\HH'$.  On the other hand, this implies that $\HH_0(\Bbbk)=\HH'(\Bbbk)$, so by Corollary \ref{cor k pts subgp} we have $\HH'\leq(\GG_0)_s\HH_0$ for some $s\in\N$.   Thus we have the following commutative diagram:
		\[
		\xymatrix{
			\GG_0/\HH_0 \ar[rr]^{i} \ar[ddrr]_{g} \ar[drr]^{h} & & \GG/\HH\ar[d]^{p} \\
			&& (\GG/\HH)^{[r]}\cong \GG_0/\HH'\ar[d]^q \\
			&& (\GG_0/\HH_0)^{[s]}
		}
		\]
		where $q$ is the projection $\GG_0/\HH'\to\GG_0/(\GG_0)_s\HH_0$ along with the isomorphism:
		\[  \GG_0/(\GG_0)_s\HH_0\cong\GG_0/(\GG_0)_s\bigg/\HH_0/(\HH_0\cap\GG_s)\cong(\GG_0/\HH_0)^{[s]}.
		\]
		Here we use that $(\GG_0)_s\cap\HH_0=(\HH_0)_s$, which follows from the exactness of the classical Frobenius functor.  All morphisms in the diagram are finite, and by Lemma \ref{lemma fr twist morphisms schemes}, $\GG_0/\HH_0$ is affine, quasi-affine, or proper if and only if $(\GG_0/\HH_0)^{[s]}$ is. Therefore we obtain that $\GG/\HH$ is correspondingly affine, quasi-affine, or proper if and only if $\GG_0/\HH_0$ is.  
		
		Finally, we show that $i$ is a universal homeomorphism. Observe that the morphisms $g$ and $p$ are universal homeomorphisms by Lemma \ref{lemma finite univ homeo}.  The maps $h$ and $q$ are affine, and induce monomorphisms on coordinate rings of any affine open.  Thus we may apply Lemma \ref{lemma univ homeo stacks} to obtain that $h$ is a universal homeomorphism, which in turn by Lemma \ref{lemma univ homeo prep} implies that $i$ is a universal homeomorphism. Note that $i$ must factor through a morphism $i_0:\GG_0/\HH_0\to(\GG/\HH)_0$, since $(\GG/\HH)_0$ is the maximal, purely even closed subscheme of $\GG/\HH$.  We obtain that $i_0$ is a universal homeomorphism and closed embedding because both $i$ and $(\GG/\HH)_0\to\GG/\HH$ are, which completes the proof.
	\end{proof}
	
	\begin{cor}\label{cor immersion}
		Suppose that $\XX$ is an algebraic scheme with an action of $\GG$, let $x\in \XX(\Bbbk)$, and write $a_x:\GG\to\XX$ for the orbit map.  Then we have an immersion $\GG/\GG_x\to\XX$, where $\GG_x$ is the stabilizer subgroup of $\GG$ (see \cite[7.2.1]{C3}).
	\end{cor}
	
	\begin{proof}
		This is \cite[Lem.~7.2.5]{C3}, which was written under the hypothesis that Theorem~\ref{Main} is valid.
	\end{proof}
	
	Observe that $\GG/\HH$ admits a left action by $\GG$.  The following corollary is clear.
	
	\begin{cor}
		Suppose that $\GG/\KK$ is isomorphic to $\GG/\HH$ as schemes with a left $\GG$-action.  Then $\HH$ is conjugate to $\KK$ inside of $\GG$.  
	\end{cor}
	
	Recall from Section \ref{sec subgrps}, we introduced the quotient groups $\GG/_{\hspace{-0.2em}gp}\NN$ for a normal subgroup $\NN\leq\GG$.  The next corollary provides a positive answer to Question~7.2.8 from \cite{C2}: is every normal subgroup of an affine group scheme the kernel of a group homomorphism?

	\begin{cor}
		Suppose that $\NN$ is a normal subgroup of an affine group scheme $\GG$ in $\cC$.
        \begin{enumerate}
            \item $\NN$ is the kernel of a group homomorphism with source $\GG$.
            \item If $\GG$ is algebraic, then $\GG\gp\NN\cong\GG/\NN$ as schemes; in other words, $\GG\gp\NN$ is a quotient in the category of schemes and the kernel of $\GG\to\GG\gp\NN$ is $\NN$.
        \end{enumerate}
        
	\end{cor}
	\begin{proof} We start by proving (2). 
		We have that $\NN_0$ is normal in $\GG_0$, and thus $\GG_0/\NN_0$ is affine, implying the same is true about $\GG/\NN$ by Theorem \ref{Main}.  Thus $\Bbbk[\GG/\NN]=\Bbbk[\GG]^{\NN}=\Bbbk[\GG\gp\NN]$. Since $\GG/\NN$ is affine, \cite[Theorem~7.3.1]{C3} shows that $\NN$ is observable in $\GG$, and hence $\NN$ is the kernel of a group homomorphism out of $\GG$ by \cite[Proposition~7.5.7]{C2}. By the universal property, $\NN$ is then more precisely the kernel of $\GG\to\GG\gp\NN$.

We have already proved (1) in case $\GG$ is algebraic. In general, we can write $\GG=\varprojlim \GG_\alpha$ for algebraic $\GG_\alpha$ and $\NN=\varprojlim\NN_\alpha$ for normal subgroups $\NN_\alpha$ of $\GG_\alpha$. Now the combination of \cite[7.5.4 and 7.5.7]{C2} demonstrates that $\NN$ is also a kernel.
	\end{proof}

    \subsection{Equivariant sheaves and induction}
	
	Recall that we write $\pi:\GG\to\GG/\HH$ for the quotient morphism.  Write $p_1,p_2:\GG\times_{\GG/\HH}\GG\to\GG$ for the two projections, and $p,a:\GG\times\HH\to\GG$ for the morphisms of projection onto $\GG$ and action of $\HH$ on $\GG$ by right multiplication.  We consider the category $\operatorname{QCoh}^\HH(\GG)$ of $\HH$-equivariant sheaves on $\GG$ with respect to this action by right multiplication.
	
	\begin{prop}\label{prop equiv qcoh shvs}
		The functors $\pi^*:\operatorname{QCoh}(\GG/\HH)\to\operatorname{QCoh}^{\HH}(\GG)$ and $(-)^{\HH}\circ\pi_*:\operatorname{QCoh}^\HH(\GG)\to \operatorname{QCoh}(\GG/\HH)$ induce equivalences of categories.
	\end{prop}
	\begin{proof}
		This result seems fairly standard, but we outline the proof.  We may apply fppf descent to the morphism $\pi:\GG\to\GG/\HH$ to obtain that $\operatorname{QCoh}(\GG/\HH)$ is equivalent to the category of quasicoherent sheaves $\FF$ on $\GG$ with an isomorphism $\alpha:p_1^*(\FF)\xto{\sim} p_2^*(\FF)$ satisfying a cocycle condition. On the other hand, by Lemma \ref{lemma quot fppf can} we have an isomorphism $\GG\times_{\GG/\HH}\GG\cong\GG\times\HH$ such that $p_1$ corresponds to $p$ and $p_2$ corresponds to $a$.  Further, the cocycle condition on $\alpha$ corresponds to the cocycle condition necessary for $\FF$ to be an $\HH$-equivariant sheaf on $\GG$.  Thus these data are equivalent, and $\pi^*:\operatorname{QCoh}(\GG/\HH)\to\operatorname{QCoh}^{\HH}(\GG)$ is an equivalence.  On the other hand, given an $\HH$-equivariant sheaf $\FF$ on $\GG$, the sections of the corresponding sheaf on $\GG/\HH$ over an open set $\UU\sub\GG/\HH$ correspond to the equalizer of the diagram
		\[
		\FF(\pi^{-1}(\UU))\rightrightarrows \FF(\pi^{-1}(\UU))\otimes\Bbbk[\HH], 
		\]
		where the two arrows are  the pullback morphisms along $p^*$ and $a^*$.  Thus we exactly obtain the $\HH$-invariants.
	\end{proof}
	
	\begin{prop}\label{prop equiv shvs}
		We have an equivalence of categories $V\mapsto \FF^V$ between $\Rep\HH$ and the category of $\GG$-equivariant quasi-coherent sheaves on $\GG/\HH$. For an open subscheme $\UU\sub\GG/\HH$, we have
		\[
		\FF^V(\UU)=(\Bbbk[\pi^{-1}(\UU)]\otimes V)^{\HH},
		\]
		where $\HH$ acts diagonally.
	\end{prop}
	
	\begin{proof}
		Observe that $\GG$ and $\GG/\HH$ admit left actions by $\GG$, and the morphism $\pi:\GG\to\GG/\HH$ is equivariant with respect to this action.  Further, the left action of $\GG$ on $\GG$ commutes with the right action of $\HH$ on $\GG$, meaning that the data of a $\GG\times\HH$ equivariant structure on a sheaf is equivalent to the data of commuting equivariant structures for $\GG$ and $\HH$.  Thus the equivalence $\pi^*$ from Proposition \ref{prop equiv qcoh shvs} defines a functor $\operatorname{QCoh}^{\GG}(\GG/\HH)\to\operatorname{QCoh}^{\GG\times\HH}(\GG)$.  Similarly, $(-)^{\HH}\circ\pi_*$ will take $\GG\times\HH$-equivariant sheaves on $\GG$ to $\GG$ equivariant sheaves on $\GG/\HH$, so that it defines a functor $\operatorname{QCoh}^{\GG\times\HH}(\GG)\to\operatorname{QCoh}^{\GG}(\GG/\HH)$.  Now we may apply Proposition \ref{prop equiv qcoh shvs} to obtain $\pi^*$ and $(-)^{\HH}\circ\pi_*$ give equivalences $\operatorname{QCoh}^{\GG}(\GG/\HH)\simeq \operatorname{QCoh}^{\GG\times\HH}(\GG)$.
		
		However by Proposition \ref{prop equiv qcoh shvs}, the latter category is equivalent to $\operatorname{QCoh}^{\HH}(\Spec\Bbbk)\simeq\operatorname{Rep}\HH$.  The description of the sections of $\FF^V$ follows from tracing through these equivalences.
	\end{proof}

	Just as in classical algebraic geometry, the category of quasicoherent sheaves on an algebraic scheme $\XX$ admits enough injectives (e.g.~\cite[Chp.~III, Exer.~3.6]{Hartshorne}), and thus one may develop the machinery of sheaf cohomology.  For a quasicoherent sheaf $\FF$ on $\XX$, write $H^\bullet(\XX,\FF)$ for the $\mathbb{N}$-graded cohomology object in $\mathscr{C}$.  
	
	On the other hand, write $\operatorname{Ind}_{\HH}^{\GG}:\operatorname{Rep}\HH\to\operatorname{Rep}\GG$ for the right adjoint to the restriction functor $\operatorname{Res}_\HH:\operatorname{Rep}\GG\to\operatorname{Rep}\HH$.  This functor is only left exact in general, and we write $R^i\operatorname{Ind}_{\HH}^{\GG}(-)$ for its right derived functors.
	
	\begin{lemma}
		We have isomorphisms $R^i\operatorname{Ind}_\HH^\GG(V)\simeq H^i(\GG/\HH,\FF^{V})$, natural in $V\in\Rep\mathcal{H}$.
	\end{lemma}
	\begin{proof}
		The proof given in \cite[Prop.~5.12]{J} works verbatim.
	\end{proof}
	
	\begin{cor}
		For a subgroup $\HH\leq\GG$, if $\GG/\HH$ is proper then $R^i\Ind_{\HH}^{\GG}(-)$ preserves compact objects.
	\end{cor}
	
	\begin{proof}
		Indeed we may apply Corollary \ref{cor proper pushforward}.
	\end{proof}

	\section{Existence of affine quotients}
	\label{Affine}
	We continue to assume that $\mathscr{C}$ satisfies (GR) and (MN1-2) in this section. Building further on the results from Section~\ref{Faisc}, we give an additional criterion for quotient schemes to (exist and) be affine when $\cC$ is semisimple. For subgroups the result reduces to the statement that $\GG/\HH$ exists and is affine if and only if $\Bbbk[\GG]$ is injective as a $\Bbbk[\HH]$-comodule.

    Let $\HH$ be an affine group scheme (not necessarily algebraic), and for the remainder of this section write $H:=\Bbbk[\HH]$.  Let $\XX=\Spec A$ be an affine scheme with a right $\HH$-action, meaning that $A$ is a right $H$-comodule such that the coaction $\rho=\rho_A:A\to A\otimes H$ is an algebra morphism.  Set $B=A^{\HH}$, a subalgebra of $A$, and $\YY=\Spec B$.

	In what follows we will write $\MM_{B}$ for the category of right $B$-modules in $\cC$, and $\MM_{A}^{H}$ for the category of $(A,H)$ Hopf modules. In geometric terms, the latter is equivalent to the category of $\HH$-equivariant quasi-coherent sheaves on $\XX$. 
	More precisely, $M\in\MM_{A}^{H}$ is a right $A$-module and right $H$-comodule such that, equivalently, $\rho_{M}:M\to M\otimes H$ is a morphism of right $A$-modules, or $a_M:M\otimes A\to M$ is a morphism of right $H$-comodules.  Here we view $M\otimes H$ as a right $A$ module via 
	\begin{equation}\label{eqn A action}
	M\otimes H\otimes A\xto{1\otimes1\otimes\rho_A}M\otimes H\otimes A\otimes H\xto{1\otimes\sigma\otimes 1}(M\otimes A)\otimes (H\otimes H)\xrightarrow{a_M\otimes m_H} M\otimes H,
	\end{equation}
    and $M\otimes H$ as right $H$-comodule via
    \begin{equation}\label{eqn H coaction}
    M\otimes H\xto{1\otimes\rho_H}M\otimes H\otimes H.
    \end{equation}

	\begin{thm}\label{thm affine quotient}
		The following are equivalent:
		\begin{enumerate}
			\item $B\to A$ is faithfully flat, and the canonical map (see Lemma~\ref{lemma quots affine})
			\[
			\XX\times \HH\to\XX\times_{\YY} \XX, \ \ \ \mathbf{can}:A\otimes_{B}A\to A\otimes H
			\]
			is an isomorphism.
			\item $\Spec B\cong\XX/_{f}\HH\cong\XX/\HH$.
			\item We have an equivalence of categories $\MM^{H}_{A}\to\MM_{B}$ given by $M\mapsto M^{\HH}$, for $M\in\MM^{H}_A$, with inverse $N\mapsto N\otimes_BA$ for $N\in\MM_B$.
		\end{enumerate} 
		If additionally $\mathscr{C}$ is semisimple, then (1)-(3) are also equivalent to:
		\begin{enumerate}
			\item[(4)] $A$ is an injective $H$-comodule, and $\mathbf{can}$ is an epimorphism in $\cC$.
		\end{enumerate}
	\end{thm}
	
	\begin{remark}
		By definition, the condition that $\mathbf{can}$ is an epimorphism is equivalent to asking that $\HH$ acts freely on $\XX$ (see Section 5).
	\end{remark}
	
	\begin{remark}
		If $\mathscr{C}$ is semisimple and satisfies (MN1-2), then $\mathscr{C}\sub\operatorname{Ver}_p$ by \cite[\S~3.2]{C2}.
		
		On the other hand, condition (4) cannot be equivalent to (1)-(3) if $\mathscr{C}$ is not semisimple.  Indeed, if $A$ is a finitely generated algebra, then $\mathbf{1}$ must split off $A$ as an object of $\mathscr{C}$.  However, if $A$ is injective as an $H$-comodule, it must be injective as an object of $\mathscr{C}$, so this would imply the unit object is injective, meaning $\mathscr{C}$ is semisimple.
	\end{remark}

	\begin{proof}[Proof of $(1)\iff(2)\iff(3)$]
		Indeed (1)$\Rightarrow$(2) is Lemma \ref{lemma quots affine} along with Lemma \ref{lemma quot exists in sch iff fais} and (2)$\Rightarrow$(3) may be proven just as in Proposition \ref{prop equiv qcoh shvs} (using fpqc descent for the general setting).  For (3)$\Rightarrow(1)$ it is clear that $B\to A$ must be  faithfully flat, while $\mathbf{can}$ is an isomorphism as this may be checked after taking $\HH$-invariants, which yields a morphism $(A\otimes_BA)^{\HH}\to A$, where the $\HH$-action is only on the right factor. That the latter is an isomorphism follows the functors in (3) being inverse.
	\end{proof}
	
	The rest of the section is devoted to proving that (4) is equivalent to (1)-(3) when $\mathscr{C}$ is semisimple.  However, we will only impose semisimplicity of $\cC$ at the very end of the proof.  We will largely follow Schneider \cite{Sch}, who works in $\mathscr{C}=\operatorname{Vec}$, or more generally in the category of modules over some commutative ring in $\operatorname{Vec}$.  We will simply generalize Schneider's arguments to the setting of more general tensor categories which works without much obstruction.  Because of this, we will try to be brief in our proofs, and point to the relevant literature on Hopf algebras in $\operatorname{Vec}$ which prove the result we need.   Finally, we remark that some version of Theorem \ref{thm affine quotient} should hold even if $A,B$ and $H$ are not commutative (which is the more general setting considered in \cite{Sch}) but we do not discuss this here.
	
	\subsection{Relative injectivity and coflatness}  We say a right $H$-comodule $V$ is \emph{relatively injective} if for any $\mathscr{C}$-split monomorphism $i:U\to W$ of $H$-comodules, every $H$-comodule morphism $U\to V$ admits an extension to $W$ via $i$.  If $\mathscr{C}$ is semisimple, then this is equivalent to asking that $V$ be injective as an $H$-comodule.  We observe that any direct summand of a relatively injective module is again relatively injective.  Further, one can verify that $M\otimes H$ is relatively injective for any $H$-comodule $M$.
	
	The following works exactly as in \cite[Thm.~1.6]{D2}; we give the formulas as presented there.
	
	\begin{lemma}\label{lemma rel inj}
		The following are equivalent:
		\begin{enumerate}
			\item $A$ is relatively injective right $H$-comodule;
			\item there exists a morphism $\phi:H\to A$ of $H$-comodules which is unital. We call $\phi$ a \emph{total integral};
			\item any Hopf module in $\MM_{A}^{H}$ is relatively injective over $H$;
			\item there exists a morphism $\varphi:A\otimes H\to A$ in $\MM_{A}^{H}$ such that $\varphi\circ\rho=\id$, see (\ref{eqn A action}), and (\ref{eqn H coaction}) for (co)module structure on $A\otimes H$.
		\end{enumerate}
	\end{lemma}
	\begin{proof}
		$(1)\Rightarrow(2)$ Because $\mathbf{1}\to H$ is split in $\mathscr{C}$ by the counit, we may use the relative injectivity of $A$ to get a unital morphism:
		\[
		\xymatrix{
			\mathbf{1}\ar[d] \ar[r] & H\ar@{-->}[dl]\\ 
			A &
		}
		\]
		
		$(2)\Rightarrow(3)$ Define $M\otimes H\to M$ by 
		\[
		a_M\circ(1\otimes \phi)\circ(1\otimes m)\circ(1\otimes S\otimes 1)\circ(\rho\otimes1).
		\]
		Then this morphism splits the coaction morphism $M\to M\otimes H$, so that $M$ is a summand of $M\otimes H$, implying it is relatively injective.
		$(3)\Rightarrow(1)$ is clear.
		
		$(2)\iff(4)$ Given $\phi:H\to A$, we set $A\otimes H\to A$ to be
		\[
		m_{A}\circ (1\otimes(\phi\circ S))\circ(1\otimes m_{H})\circ (\rho\otimes 1).
		\]
		Conversely, given $\psi:A\otimes H\to A$, define $H\to A$ by $\psi\circ (\eta\otimes\id)$.     
	\end{proof}

	Recall that given a right $H$-comodule $V$ and a left $H$-comodule $W$, we may form the cotensor product $V\square_HW$ defined by the equalizer:
	\begin{equation}\label{eq1}
		V\square_HW\to V\otimes W\rightrightarrows V\otimes H\otimes W.
	\end{equation}
	If $W$ is also a right $H$-comodule (e.g.~$W=H$), then $V\square_HW$ admits the structure of a right $H$-comodule via viewing (\ref{eq1}) in the category of right $H$-comodules.  For a right $H$-comodule $V$, write $V^{tw}$ for the left $H$-comodule on the same underlying object $V$ but with coaction:
	\[
	V\to V\otimes H\xto{\sigma}H\otimes V\xto{s_H\otimes 1}H\otimes V,
	\]
	where $s_H$ is the antipode.  The following is proven in the case of $\mathscr{C}=\operatorname{Vec}$ in, for instance, \cite[Lem.~3.1]{Sch}. 
	
	\begin{lemma}\label{lemma twist comodules}
		The functor $(-)^{tw}$ is an equivalence of categories between right and left $H$-comodules.  Further, for any two right $H$-comodules $V,W$, we have
		\[
		V\square_HW^{tw}\cong (V\otimes W)^{\HH}.
		\]
	\end{lemma}
	
	We say that a right $H$-comodule $V$ is coflat if the functor $V\square_H(-)$ is exact.
	\begin{lemma}[Cor.~3.2 of \cite{Sch}]
		The following are equivalent:
		\begin{enumerate}
			\item $A$ is a coflat $H$-comodule;
			\item $\MM^{H}_{A}\to\MM_B$, $M\mapsto M^{\HH}$ is exact.
		\end{enumerate}
	\end{lemma}
	\begin{proof}
		$(1)\Rightarrow(2)$ This functor is always left exact.  On the other hand, since $M\square_H\mathbf{1}=M^{\HH}$ is a submodule of the left-hand side, we have an epimorphism
		\[
		\psi:M\square_{H}A^{tw}\to M^{\HH},
		\]
		given by the module action of $A$ on $M$.  Since $(-)^{tw}$ is an equivalence of categories between right and left $H$-comodules, $A^{tw}$ will be a coflat left comodule, so that $M\mapsto M\square_{H} A^{tw}$ is exact by assumption.  From this a commutative diagram shows that $(-)^{\HH}$ will also preserve epimorphisms.
		
		$(2)\Rightarrow(1)$ For any right $H$-comodule $V$ we have that $V\mapsto (V\otimes A)^{\HH}$ is exact.  However by Lemma \ref{lemma twist comodules}, $(V\otimes A)^{\HH}\cong V\square_{H}A^{tw}$, so this implies that $A^{tw}$ is coflat, which is equivalent to $A$ being coflat.
	\end{proof}
	In the following, note that $(-)\otimes_BA$ is left adjoint to $(-)^{\HH}$.
	\begin{lemma}[Lem.~3.4 of \cite{Sch}]\label{lemma rel inj implies iso}
		If $A$ is a relatively injective right $H$-comodule, then the unit morphism $M\to(M\otimes_BA)^{\HH}$ is an isomorphism.
	\end{lemma}
	
	\begin{proof}
		Indeed, it suffices to show that the canonical map
		\[
		c:M\otimes_{B}(A\square_{H}\Bbbk)\to(M\otimes_{B}A)\square_{H}\Bbbk
		\]
		is an isomorphism.  Let $\phi:H\to A$ be a total integral, and define a trace map $\operatorname{tr}:A\to B$ by 
		\[
		\operatorname{tr}=m_{A}\circ(1\otimes\phi)\circ(1\otimes S)\circ\Delta.
		\]
		Then an inverse to $c$ is given by $1\otimes\operatorname{tr}$.
	\end{proof}
	
	The following is essentially \cite[Thm.~3.5]{Sch}.
	\begin{thm}\label{thm rel inj eq}
		Suppose that $A$ is a relatively injective $H$-comodule and either:
		\begin{enumerate}
			\item $\mathbf{can}$ is an isomorphism; or,
			\item $\mathbf{can}$ is an epimorphism and $\mathscr{C}$ is semisimple.
		\end{enumerate}   
		Then the induction functor $\MM_{B}\to\MM_{A}^{H}$ is an equivalence.
	\end{thm}
	\begin{proof}
		By Lemma \ref{lemma rel inj implies iso}, the unit morphism of the adjunction $M\to (M\otimes_{B}A)^{\HH}$ is an isomorphism for $M\in\MM_{B}$.  Thus it remains to check that the other adjunction, $A\otimes_{B}N^{\HH}\to N$ is also an isomorphism for $N\in\MM_{A}^{H}$.    
		
		For this, choose $\varphi:A\otimes H\to A$ in $\MM^H_A$ which splits $\rho$, see Lemma~\ref{lemma rel inj}. Consider $f:N\otimes A\otimes H\to N$, given by 
		\[
		f=m\circ(1\otimes\varphi)\circ(1\otimes1\otimes m)\circ(1\otimes1\otimes S\otimes 1)\circ(1\otimes\sigma\otimes 1)\circ(\rho\otimes1\otimes1).
		\]
		We view $N\otimes A\otimes H$ as an $(A,H)$-Hopf module where $A$ acts via the $A\otimes H$ factor, and $H$ coacts via the $H$-factor  as in \eqref{eqn H coaction}.  Then $f$ is a Hopf module morphism, and is split in $\mathscr{C}$ via $(\id_N\otimes \eta_A\otimes\id_H)\circ\rho_N$.  
        
		Now by our assumption on $\mathbf{can}$, we deduce that the composition $g=f\circ (\id_N\otimes \mathbf{can}):N\otimes A\otimes_BA\to N\otimes A\otimes H\to N$ is a split surjection.  We may view $N\otimes A\otimes_BA$ as an $(A,H)$-Hopf module via the rightmost factor $A$.  Then $g$ is a morphism of Hopf modules, so its kernel $K$ is once again a Hopf module.  By Lemma \ref{lemma rel inj}, we obtain that $K$ is a relatively injective $H$-comodule, so that $g$ is split as a $H$-comodule morphism.  
		
		Write $M=N\otimes A\otimes_{B}A$.  Then we claim the adjunction morphism $M\to M^{\HH}\otimes_{B}A$ is an isomorphism.  Indeed, we have $(N\otimes A\otimes_{B} A)^{\HH}\cong N\otimes A$ by our other adjunction, and from this it is clear.  Since $N$ is a direct summand of $M$ as an $H$-comodule, the adjunction is also an isomorphism for it, and we are done.
	\end{proof}

	We may now complete the proof of Theorem \ref{thm affine quotient}:
	\begin{proof}[Conclusion of Proof of Theorem~\ref{thm affine quotient}]
		We already proved $(4)\Rightarrow(3)$ (in Theorem~\ref{thm rel inj eq}) and $(1)\iff(2)\iff(3)$.  Thus it remains to show that $(1)\Rightarrow(4)$.  Because $\mathscr{C}$ is semisimple, $A$ is relatively injective if and only if it is injective if and only if it is coflat (where, as in the classical case, the latter equivalence reduces to the case of coalgebras in $\Cs_{fin}$). For a right $H$-comodule $V$, we have isomorphisms
		\[
		A\otimes_B(A\square_{H}V)\cong (A\otimes_BA)\square_{H}V\cong(A\otimes H)\square_HV\cong A\otimes(H\square_HV)\cong A\otimes V.
		\]
		Here, we have applied \cite[Prop.~1.3]{Tak} for the first isomorphism, our assumption that $\mathbf{can}:A\otimes_BA\to A\otimes H$ is an isomorphism for the second isomorphism, and the isomorphism $H\square_HV\cong V$ which holds for any comodule.  Since $B\to A$ is faithfully flat, we learn that $A\square_H-$ is exact, meaning that $A$ is coflat, which completes our proof.
	\end{proof}

	\section{Finer description of homogeneous spaces for $\operatorname{Ver}_p$}\label{Verp}
	
	We now specialize to $\mathscr{C}=\operatorname{Ver}_p$, which we assume for the remainder of this paper.  We will give another proof of the existence of quotients in $\operatorname{Ver}_p$ by mimicking the proof for $\operatorname{sVec}$ from \cite{MT}.
	This is an important special case, since by the main theorem of \cite{CEO2}, it is the only moderate growth, incompressible tensor category (together with its tensor subcategories) that is semisimple (more generally: Frobenius exact).
	We will provide a more explicit description of homogeneous spaces using the semisimplicity of $\mathscr{C}$.  In what follows we write $\g$ for the Lie algebra of $\GG$.
    
    We begin by recalling the following theorem, which is proven in \cite[Lem.~7.15]{V1} and \cite[Thm.~5.1]{M}. Recall the notation $X_{\neq 0}$ from Example~\ref{example ver_p}. 
    
    \begin{thm}
    We have an isomorphism of $\GG_0$-modules and algebras:
        \[
	\Bbbk[\GG]\cong \Bbbk[\GG_0]\otimes S\g_{\neq0}^*,
	\]
    where $\GG_0$ acts by infinitesimal right translation on both $\Bbbk[\GG]$ and $\Bbbk[\GG_0]$, and $\GG_0$ acts on $\g_{\neq0}$ via the adjoint action.
    \end{thm}

Let $\HH\leq\GG$ be a subgroup, and write $p_0:\GG_0\to\GG_0/\HH_0$ for the quotient morphism which we know exists in $\Vec$.  We would like to construct a quotient $\GG/\HH$ with $(\GG/\HH)_0=\GG_0/\HH_0$; in particular, $\GG/\HH$ and $\GG_0/\HH_0$ will have the same underlying topological spaces.  To do this, we will construct a quotient affine locally, and glue. 

The following lemma is implicitly used in \cite{MT}.  For an open subscheme $\VV\sub \GG_0$, write $\widetilde{\VV}\sub\GG$ for the open subscheme of $\GG$ it determines.  Observe that if $\VV$ is stable under translation by $\HH_0$ on the right, then $\widetilde{\VV}$ will be stable under translation by $\HH$ on the right.

	\begin{lemma}
		Suppose that for every affine open subscheme $\UU\sub\GG_0/\HH_0$ there exists an affine quotient $\VV_{\UU}\cong \widetilde{p_0^{-1}(\UU)}/\HH$ with $(\VV_\UU)_0=\UU$.  Suppose further that whenever $\UU'\sub\UU$ is an inclusion of affine opens, the natural map $\VV_{\UU'}\to\VV_{\UU}$ is the open embedding of $\VV_{\UU'}$ into $\VV_{\UU}$ determined by $\UU'\sub\UU$.  Then we may glue the schemes $\{\VV_\UU\}$ to construct a quotient $\GG/\HH$.
	\end{lemma}
	\begin{proof}
		Under these assumptions we may glue the schemes $\{\VV_\UU\}$ using the universal property of the quotient to check that the gluing conditions hold.
	\end{proof}

 We now write
	\[
	W_{\GG}:=(\m_{\GG}/\m_{\GG}^2)_{\neq0}\cong(\g_{\neq0})^*, \ \ \ \ W_{\HH}:=(\m_{\HH}/\m_{\HH}^2)_{\neq0}\cong(\h_{\neq0})^*.
	\]
    The natural quotient map $\Bbbk[\GG]\to\Bbbk[\HH]$ induces an epimorphism $\g^*\to\h^*$. Thus we may define $Z\in\Rep\HH_0$ by the following SES of $\HH_0$-modules:
	\[
	0\to Z\to W_{\GG}\to W_{\HH}\to 0.
	\]
    
	Let $\UU\sub \GG_0/\HH_0$ be an affine open which we fix for the remainder of the paper, and consider $p_0^{-1}(\UU)\sub \GG_0$.  Write $B=\Bbbk[\UU]$ and $A:=\Bbbk[p_0^{-1}(\UU)]$.  We would like to describe $\OO_{\GG/\HH}$ locally on $\UU$ in terms of $A$ and $Z$.  To this end, we consider the algebra $S_A(Z)=A\otimes S(Z)$ which lives in $\MM_{A}^{\Bbbk[\HH_0]}$ (see Section 8 for notation).  Then we set:
	\[
	\A:=S_A(Z)\square_{\Bbbk[\HH_0]}\Bbbk[\HH],
	\]
	and
	\[
	\mathbb{B}:=S_A(Z)^{\HH_0}=S_A(Z)\square_{\Bbbk[\HH_0]}\Bbbk.
	\]
	The following mimics \cite[Prop.~4.2]{MT} and is proven verbatim.
	\begin{prop}\label{prop A and B}
		\begin{enumerate}
			\item $\mathbb{A}$ is a finitely generated algebra, is injective as an $\HH$-module, and  we have $\mathbb{A}^{\HH}\cong\mathbb{B}$.
			\item $\mathbb{B}$ is an $\N$-graded subalgebra of $S_A(Z)$ with $\mathbb{B}_0=B$, and $\mathbb{B}_1=(A\otimes Z)^{\HH_0}$ is a finitely generated projective $B$-module with every fibre isomorphic to $Z$.  Further, the natural map
			\[
			S_{B}(\mathbb{B}_{1})\to \mathbb{B}
			\]
			is an isomorphism of algebras.  In particular, $\mathbb{B}$ is also Noetherian.
		\end{enumerate}
	\end{prop}
	
	\begin{lemma}\label{lemma splitting compatible}
		The inclusion $A\otimes Z\hookrightarrow A\otimes W_{\GG}$ splits in $\MM_{A}^{\Bbbk[\HH_0]}$.  Further, if $\mathcal{U}'\sub\UU\sub\GG_0/\HH_0$ is an affine open subscheme of $\UU$, and we set $A':=\Bbbk[\pi_0^{-1}(\UU')]$ then this splitting can be chosen compatibly for both $A$ and $A'$.  
	\end{lemma}
	
	\begin{proof}
		We follow \cite[Thm.~1.7 (ii)]{D2}.  By Lemma \ref{lemma rel inj}, there exists a total integral $\phi:\Bbbk[\HH_0]\to A$, and a total integral for $A'$ may be chosen by postcomposing $\phi$ with the natural morphism $A\to A'$.   Choose a splitting $W_{\GG}\to Z$ in $\operatorname{Ver}_p$ of the inclusion $Z\sub W_{\GG}$, and use this to induce a splitting $p:A\otimes W_{\GG}\to A\otimes Z$ of $A$-modules.  Set $q=\lambda_{A\otimes W_{\GG}}\circ(p\otimes 1)\circ\rho_{A\otimes W_{\GG}}$, where $\lambda_M:M\otimes A\to M$ is given by 
		\[
		\lambda_M=m_{A}\circ(1\otimes\phi)\circ (1\otimes m_{\Bbbk[\HH]})\circ(1\otimes S\otimes 1)\circ(\rho_M\otimes 1).
		\]
		Then one can check that $q$ is our desired splitting.  Since $q$ only depends on a choice of total integral $\phi:\Bbbk[\HH_0]\to A$ and our splitting $W_{\GG}\to Z$, and these can be chosen compatibly for $A$ and $A'$ via the natural map $A\to A'$, we have the desired compatibility as well.
	\end{proof}
	
	Using Lemma \ref{lemma splitting compatible}, choose a retraction $A\otimes W_{\GG}\to A\otimes Z$ in $\MM_{A}^{\Bbbk[\HH_0]}$.  Define $\psi:\Bbbk[\GG]\to \A$ by:
	\[
	\Bbbk[\GG]\xto{\Delta_{\GG}}\Bbbk[\GG]\square_{\Bbbk[\HH_0]}\Bbbk[\HH]\cong (\Bbbk[\GG_0]\otimes S(W_{\GG}))\square_{\Bbbk[\HH_0]}\Bbbk[\HH]\to(A\otimes S(Z))\square_{\Bbbk[\HH_0]}\Bbbk[\HH]=\A
	\]
	
	\begin{lemma}
		The morphism $\psi$ corresponds to an $\HH$-equivariant open embedding of schemes $\Spec\A\hookrightarrow \GG$ corresponding to the open embedding $\widetilde{\pi_0^{-1}(\UU)}\sub \GG$.
	\end{lemma}
	
	\begin{proof}
		Pass to the associated graded as in \cite[Prop.~4.8]{MT}, and use Proposition \ref{lemma str sheaf desc} and Lemma \ref{lemma open emb}.
	\end{proof}

 In the following lemma, for a commutative algebra $R$ we write $\operatorname{gr}R=\bigoplus\limits_{i}J_R^i/J_R^{i+1}$.
	\begin{lemma}\label{lemma open emb}
		Suppose that $\phi^*:R\to S$ is a morphism of algebras such that $J_R$ is nilpotent.  Then the induced map $\phi:\Spec S\to\Spec R$ is an open immersion if and only if $\operatorname{gr}\phi:\Spec(\operatorname{gr}S)\to\Spec(\operatorname{gr}R)$ is an open immersion.  The latter condition holds if  we have both
		\begin{enumerate}
			\item $\Spec S/J_{S}\to\Spec R/J_{R}$ is an open immersion; and, 
			\item for every $\p\in\Spec S$, the map on stalks $\operatorname{gr}R_{\p}\to\operatorname{gr}S_{\p}$ is an isomorphism.
		\end{enumerate}
	\end{lemma}
	
	\begin{proof}
		If $\phi$ is an open immersion, it is clear that $\operatorname{gr}\phi$ is as well.  Conversely, if $\operatorname{gr}\phi$ is an open immersion then there exists elements $\{f_i\in R/J_R:i\in I\}$ inducing isomorphisms 
\begin{equation}\label{eqn iso}
    (\operatorname{gr}R)_{f_i}\cong (\operatorname{gr}S)_{f_i},
\end{equation} and $\Spec(\operatorname{gr} S)=\bigcup\limits_{i\in I}\Spec (\operatorname{gr}S)_{f_i}$.  If we choose lifts $f_i'\in R_{(0)}$ of the elements $f_i$, then the natural maps $R_{f_i'}\to S_{f_i'}$ will induce the isomorphisms (\ref{eqn iso}) on the associated graded.  Since $J_R$ is a nilpotent ideal, this isomorphism must reflect an isomorphism of algebras $R_{f_i'}\cong S_{f_i'}$, and from here it is easy to deduce that $\operatorname{gr}\phi$ is an open immersion.  The equivalence of $\operatorname{gr}\phi$ being an open immersion with conditions (1) and (2) is straightforward.\end{proof}
		
	\begin{lemma}
		The morphism $\Spec\A\to \Spec\mathbb{B}$ realizes $\Spec\mathbb{B}\cong(\Spec\A)/\HH$. 
	\end{lemma}
	
	\begin{proof}
		By Theorem \ref{thm affine quotient}, it suffices to check that $\A$ is an injective $\HH$-module and that our canonical morphism is a surjection.  We showed that $\A$ is injective in Proposition \ref{prop A and B}, so we only need to show that the canonical morphism 
		\[
		\A\otimes_{\mathbb{B}}\A\to\A\otimes\Bbbk[\HH]
		\]
		is an { epimorphism}.  However, this follows from the fact that $\HH$ acts freely on $\GG$, meaning that $\GG\times\HH\to \GG\times \GG$ is a closed embedding. 
	\end{proof}
	
	\begin{lemma}
		We have compatibility of quotients upon restriction to open subschemes.
	\end{lemma}
	
	\begin{proof}
		This follows from the fact that we could choose our splittings compatibly, see Lemma \ref{lemma splitting compatible}.
	\end{proof}

    \ We have now given another proof of the existence of $\GG/\HH$ in $\operatorname{Ver}_p$.  However, we have done more by giving an explicit local construction of this quotient scheme.  In particular, we learn the following. 
	
	\begin{thm}\label{thm ver p quotient}
		Let $\GG$ be an algebraic group, and let $\HH\leq\GG$ be a closed subgroup. Then:
		\begin{enumerate}
			\item  
			we have an identification: $(\GG/\HH)_0=\GG_0/\HH_0$;
			\item $\JJ_{\GG/\HH}/\JJ_{\GG/\HH}^2$ is the $\GG_0$-equivariant sheaf on $\GG_0/\HH_0$ with fibre  $Z=\operatorname{ker}(W_{\GG}\to W_{\HH})$;
			\item $\GG/\HH$ is affine locally isomorphic to the scheme $(|\GG_0/\HH_0|,S(\FF^{Z}))$ (see Proposition \ref{prop equiv shvs} for meaning of $\FF^Z$).
		\end{enumerate}
	\end{thm}
	\begin{proof}
		The proofs of (1) and (2) follow from our work above, while (3) is explained in \cite[Rmk.~4.3]{MT}.
	\end{proof}

\begin{remark}
    Theorem \ref{thm ver p quotient}, especially part (3), has important implications in representation theory, as it tells us that the associated graded of a $\GG$-equivariant sheaf on $\GG/\HH$ is given by a directly computable $\GG_0$-equivariant sheaf on $\GG_0/\HH_0$.  This fact has been used heavily in the computation of characters of irreducible representations of supergroups, as briefly discussed in the introduction.  
\end{remark}

	\bibliographystyle{amsalpha}
	%\bibliography{bibliography}

	\textsc{\footnotesize K.C.: School of Mathematics and Statistics, University of Sydney, NSW 2006, Australia} 
	
	\textit{\footnotesize Email address:} \texttt{\footnotesize kevin.coulembier@sydney.edu.au}
	
	\textsc{\footnotesize A.S.: School of Mathematics and Statistics, University of Sydney, NSW 2006, Australia} 
	
	\textsc{\footnotesize School of Mathematics and Statistics, UNSW Sydney, NSW 2052,
		Australia}
	
	\textit{\footnotesize Email address:} \texttt{\footnotesize xandersherm@gmail.com}

\end{document}